\documentclass{article}

\usepackage[T1]{fontenc}
\usepackage[utf8]{inputenc}

\usepackage[margin=1.4in]{geometry}
\usepackage{graphicx}
\usepackage{empheq}

\usepackage{amssymb,amsthm}
\usepackage{mathtools}
\numberwithin{equation}{section}
\usepackage{bbold}

\usepackage[
style=ieee,
citestyle=numeric-comp,
sorting=nyt,
url=false,
doi=false,
sortcites,
defernumbers,
backref,
backend=biber
]{biblatex}
\usepackage{hyperref}

\usepackage{todonotes}
\usepackage{url}

\usepackage[nameinlink, capitalise, noabbrev]{cleveref}

\usepackage{xfrac}
\usepackage{nicefrac}

\newtheorem{theorem}{Theorem}
\newtheorem{proposition}{Proposition}[section]
\newtheorem{lemma}[proposition]{Lemma}
\newtheorem{corollary}[proposition]{Corollary}
\theoremstyle{remark}
\newtheorem{remark}[proposition]{Remark}

\theoremstyle{definition}

\newtheorem{definition}{Definition}
\newtheorem{assumption}{Assumption}

\numberwithin{equation}{section}
\allowdisplaybreaks[4]

\newcommand{\R}{\mathbb{R}}
\newcommand{\B}{\mathfrak{B}}
\newcommand{\N}{\mathbb{N}}
\newcommand{\X}{\mathcal{X}}
\newcommand{\Y}{\mathcal{Y}}

\DeclareMathOperator{\E}{\mathbb{E}}
\newcommand{\eps}{\varepsilon}
\renewcommand{\epsilon}{\varepsilon}
\renewcommand{\rho}{\varrho}
\newcommand{\crho}{C_\rho}

\graphicspath{{./figures/}}

\DeclareMathOperator*{\esssup}{ess\,sup}
\DeclareMathOperator*{\essinf}{ess\,inf}
\DeclareMathOperator*{\argmin}{arg\,min}
\DeclareMathOperator*{\argmax}{arg\,max}
\DeclareMathOperator{\dist}{dist}
\DeclareMathOperator{\sdist}{sdist}

\newcommand{\norm}[1]{\left\|#1\right\|}
\newcommand{\abs}[1]{\left|#1\right|}
\newcommand{\st}{\,:\,}
\newcommand{\de}{\,\mathrm{d}}
\newcommand{\one}{\mathbf{1}}
\DeclareMathOperator{\supp}{supp}
\newcommand{\tube}{\Omega'}

\newcommand{\comment}[1]{}

\numberwithin{equation}{section}

\newcommand{\closure}[1]{\overline{#1}}

\newcommand{\grad}{\nabla}
\DeclareMathOperator{\tmpdiv}{div}
\renewcommand{\div}{\tmpdiv}

\renewcommand{\P}{\mathcal{P}}
\renewcommand{\L}{\mathcal{L}}
\renewcommand{\H}{\mathcal{H}}
\DeclareMathOperator{\Lip}{Lip}

\DeclareMathOperator{\Per}{Per}
\DeclareMathOperator{\TV}{TV}

\def\Xint#1{\mathchoice
{\XXint\displaystyle\textstyle{#1}}%
{\XXint\textstyle\scriptstyle{#1}}%
{\XXint\scriptstyle\scriptscriptstyle{#1}}%
{\XXint\scriptscriptstyle\scriptscriptstyle{#1}}%
\!\int}
\def\XXint#1#2#3{{\setbox0=\hbox{$#1{#2#3}{\int}$ }
\vcenter{\hbox{$#2#3$ }}\kern-.6\wd0}}
\def\intbar{\Xint-}

\definecolor{myblue}{rgb}{0.678 0.847 0.902}
\newcommand*\mybluebox[1]{%
\colorbox{myblue}{\hspace{1em}#1\hspace{1em}}}

\begin{document} 

\title{
A mean curvature flow arising in adversarial training}
\author{Leon Bungert\thanks{Institute of Mathematics \& Center for Artificial Intelligence and Data Science (CAIDAS), University of Würzburg, Emil-Fischer-Str. 40, 97074 Würzburg, Germany. Email: \href{mailto:leon.bungert@uni-wuerzburg.de}{leon.bungert@uni-wuerzburg.de}
} 
\and 
Tim Laux\thanks{Faculty of Mathematics, University of Regensburg, Universit{\"a}sstra{\ss}e 31, 93053 Regensburg, Germany. Email: \href{mailto:tim.laux@ur.de}
{tim.laux@ur.de}}
\and 
Kerrek Stinson
\thanks{Hausdorff Center for Mathematics, University of Bonn, Endenicher Allee 62, 53115 Bonn, Germany. Email: 
\href{mailto:kerrek.stinson@hcm.uni-bonn.de}{kerrek.stinson@hcm.uni-bonn.de}}
}
\maketitle

\begin{abstract}
    We connect adversarial training for binary classification to a geometric evolution equation for the decision boundary. 
    Relying on a perspective that recasts adversarial training as a regularization problem, we introduce a modified training scheme that constitutes a minimizing movements scheme for a nonlocal perimeter functional. 
    We prove that the scheme is monotone and consistent as the adversarial budget vanishes and the perimeter localizes, and as a consequence we rigorously show that the scheme approximates a weighted mean curvature flow.
    This highlights that the efficacy of adversarial training may be due to locally minimizing the length of the decision boundary. 
    In our analysis, we introduce a variety of tools for working with the subdifferential of a supremal-type nonlocal total variation and its regularity properties.
    \\
    \\
    \textbf{Keywords:} mean curvature flow, adversarial training, adversarial machine learning, minimizing movements, monotone and consistent schemes
    \\
    \textbf{AMS subject classifications:} 28A75, 35D40, 49J45, 53E10, 68T05  
\end{abstract}

\section{Introduction}
In the last decade, machine learning algorithms and in particular deep learning have experienced an unprecedented success story.
Such methods have proven their capabilities, inter alia, for the difficult tasks of image classification and generation. 
Most recently, the advent of large language models is expected to have a strong impact on various aspects of society.

At the same time, the success of machine learning is accompanied by concerns about the reliability and safety of its methods.
Already more than ten years ago it was observed that neural networks for image classification are susceptible to adversarial attacks \cite{szegedy2013intriguing}, meaning that imperceptible or seemingly harmless perturbations of images can lead to severe misclassifications.
As a consequence, the deployment of such methods in situations that affect the integrity and safety of humans, e.g., for self-driving cars or medical image classification, is risky.

To mitigate these risks, the scientific community has been developing different approaches to robustify machine learning in the presence of potential adversaries. 
The most prominent of these approaches in the context of classification tasks is \emph{adversarial training} \cite{goodfellow2014explaining,
madry2017towards}, which is a robust optimization problem of the form
\begin{align}\label{eq:AT_general}
	\inf_{u\in\H} \E_{(x,y)\sim\mu}\left[\sup_{\tilde x \in B(x,\eps)}\ell(u(\tilde x),y)\right],
\end{align}
where we use the notation $\E_{z\sim\mu}[f(z)]:=\int f(z)\de\mu(z)$.
The ingredients of adversarial training are readily explained: 
The probability measure $\mu \in \P(\X\times\Y)$ models the distribution of given training pairs in the so-called feature space $\X$ and the label space $\Y$.
Here $\X$ is a metric space, and $\Y$ a set, e.g., $\Y=\{0,\dots,K-1\}$ describing $K\in\N$ classes.
In realistic situations, one uses an empirical distribution of the form $\mu = \frac{1}{M}\sum_{i=1}^M\delta_{(x_i,y_i)}$ where $(x_i,y_i)\in\X\times\Y$ for $i=1,\dots,M$. 
The optimization takes place in a so-called hypothesis class $\H$ which is nothing but a class of functions from $\X$ to $\Y$, e.g., linear functions, measurable functions, parametrized neural networks, etc.
We let $\ell:\Y\times\Y\to\R$ be a so-called loss function, which is often chosen as a power of a norm or $f$-divergence.
Finally, in essence, the optimal classifier $u$ should satisfy $u(\tilde x)\approx y$ for $\mu$-almost every $(x,y)\in\X\times\Y$ and all $\tilde x\in B(x,\eps)$.
Thereby, \labelcref{eq:AT_general} enforces robustness of the classification in $\eps$-balls around the data points, where $\eps>0$ is called the \emph{adversarial budget}.

Already in \cite{madry2017towards} it has been empirically observed that \labelcref{eq:AT_general} indeed allows one to compute neural networks that are significantly more robust than those trained with the standard approach (corresponding to $\eps=0$ in \labelcref{eq:AT_general}).
However, the mathematical understanding of adversarial training and related problems only began growing a few years ago:
One line of research connects \labelcref{eq:AT_general} with (multimarginal) optimal transport or distributionally robust optimization problems \cite{pydi2020adversarial,pydi2021many,trillos2022multimarginal} and uses tools from these disciplines to analyse adversarial training.
Existence of solutions to \labelcref{eq:AT_general} in the binary classification case where $\Y=\{0,1\}$, $\ell$ is the $0$-$1$ loss $\ell(\tilde y,y)=\abs{\tilde y-y}$, and $\H$ is a class of measurable functions was proved in \cite{awasthi2021existence,bungert2023geometry}.
In \cite{awasthi2021existence}, the authors consider closed balls $B(x,\eps)$ in \labelcref{eq:AT_general} and work with classifiers which are characteristic functions of universally measurable sets in $\R^N$.
In contrast, in \cite{bungert2023geometry} open balls are used and the classifiers are characteristic functions of Borel measurable subsets of a generic metric measure space.
The authors of \cite{bungert2023geometry} also proved that adversarial training for binary classification is equivalent to the following variational regularization problem:
\begin{align}\label{eq:var_reg}
    \inf_{A\in\B(\X)}
    \mathbb E_{(x,y)\sim\mu}\left[\abs{\one_A(x)-y}\right]
    +
    \eps\Per_\eps(A).
\end{align}
Here, $\Per_\eps$ is a non-local and data-dependent perimeter functional that regularizes the decision boundary $\partial A$ between the two classes.
A similar decomposition into a ``natural error'' and a ``boundary error'' was studied in \cite{zhang2019theoretically} and used to derive the TRADES algorithm, which essentially replaces the regularization parameter $\eps$ in \eqref{eq:var_reg} by $\tfrac{1}{\lambda}$ for $\lambda>0$.
Also for other notions of robustness which are weaker than adversarial robustness, geometric interpretations similar to \eqref{eq:var_reg} exists, see \cite[Section 4]{bungert2023geometry} or \cite{bungert2023begins}.
We also remark that generalizations of some of the results in \cite{awasthi2021existence,bungert2023geometry} to the case of multi-class classification can be found in \cite{trillos2023existence}.
Finally, an overview of recent mathematical developments in the field can be found in \cite{garcia2023analytical}.

The perspective in \labelcref{eq:var_reg} opens the door for the geometric analysis of adversarially robust classifiers. 
As a first step in this direction, in \cite{bungert2023geometry} it was shown that maximal and minimal minimizers of \labelcref{eq:var_reg} possess one-sided regularity properties, and that for $\X=\R^N$ there exists a solution with a boundary that is locally the graph of a $C^{1,1/3}$ function. 
Subsequently, it was shown in \cite{bungert2022gamma} that (even for discontinuous densities with bounded variation) the nonlocal perimeter $\Per_\eps$ Gamma-converges to a weighted local perimeter and, as a consequence, solutions of \labelcref{eq:var_reg} converge to perimeter-minimal solutions of \labelcref{eq:var_reg} with $\eps=0$. 
In \cite{trillos2022adversarial} it was shown that for sufficiently small $\eps>0$ adversarially robust classifiers evolve (as parametrized by $\eps$) according to a geometric flow, when smooth solutions starting from the Bayes classifier exist. 
Expansions used to derive this result show that, infinitesimally in $\eps$, this flow is a weighted mean curvature flow, which shows that adversarial training is connected to decreasing the length of the decision boundary.

The \textbf{main contribution} of the present paper is to make this connection with mean curvature flow rigorous in a general setting and to move beyond the short time regime of \cite{trillos2022adversarial}.
To achieve this, we will introduce a slight modification of the adversarial training problem \labelcref{eq:var_reg}. 
Intuitively, the proposed iterative scheme prepares for attacks by an adversary with \emph{total adversarial budget} $T>0$ and (instantaneous) adversarial budget $\epsilon > 0$, allowing the adversary to corrupt the data on scale $\epsilon$ and even to react to modified classifiers at most $T/\eps$ times. 
As we will see in \cref{sec:MMS} below, the scheme can be interpreted as a minimizing movements scheme for mean curvature flow, in the spirit of Almgren--Taylor--Wang \cite{almgren1993curvature}. 
To select unique solutions we consider a strongly convex Chambolle-type scheme \cite{chambolle2004algorithm} and prove that it is monotone and consistent with respect to a weighted mean curvature flow, thereby proving convergence of the scheme to smooth flows (\cref{thm:main}).

The \textbf{main challenge} and the reason why our results are not just straightforward extensions of existing ones is that the adversarial budget $\eps>0$ in \labelcref{eq:var_reg} acts both as a time step and as a non-locality parameter for the perimeter $\Per_\eps$. 
Hence, in order to prove consistency with mean curvature flow, we have to perform a careful analysis of the associated total variation functional and its subdifferential, showing that the latter is consistent with the $1$-Laplace operator for a suitable class of functions.

We would like to emphasize that adversarial training is not the only method in data science connected to mean curvature flow. 
In particular, in the field of graph-based learning the so-called Merriman--Bence--Osher (MBO) algorithm has been employed frequently for clustering data sets or solving semi-supervised learning problems, see, e.g.,~\cite{merkurjev2013mbo,merkurjev2018semi,van2014mean,calder2020poisson}.
For rigorous connections of such approaches to mean curvature flow we refer to \cite{laux2022large,laux2023large}.

\textbf{Organization of the paper.} 
The rest of the paper is organized as follows. 
In the next section we precisely introduce the proposed adversarial training scheme and state our main result---convergence of the method to weighted mean curvature flow. In \cref{sec:totalVarProp}, we deduce the needed properties for the nonlocal total variation and, in particular, study its subdifferential. 
Finally, in \cref{sec:convergence} we prove convergence of the adversarial training scheme by verifying that it is monotone and consistent with respect to weighted mean curvature flow.

\textbf{Notation.} For the reader's convenience, we collect notation used throughout the paper here. Typically, $\Omega\subset \R^N$ will be a bounded domain (i.e., a non-empty, open, and connected set).
We use $\L^N$ to denote the $N$-dimensional Lebesgue measure and $\B(\Omega)$ to denote the Borel measurable subsets of $\Omega$.
Furthermore, we use $\abs{\cdot}$ for the Euclidean norm of a vector in $\R^N$ and $\mathbb{1}$ for the $N\times N$ identity matrix. For a set $A \subset \R^N$, we let $\one_A$ be the characteristic function taking the value $1$ on $A$ and $0$ otherwise. For any set $\Omega\subset \R^N$, we define the inner parallel set of distance $a>0$ through
\begin{equation}\label{eqn:inner_parallel}
    \Omega_a : = \{x\in \Omega \st \dist(x,\R^N\setminus\Omega)>a\}.
\end{equation}
Finally, for $x\in \R^N$ and $\eps>0$, we denote open balls by 
$B(x,\eps) := \{y\in\R^N\st \abs{x-y}<\eps\}$.

\section{From adversarial training to mean curvature flow}

Let $\Omega\subset\R^N$ be a bounded domain, $\mu\in \P(\Omega\times \{0,1\})$ be a probability measure, and $\ell(\bar y, y) = 1_{\bar y\neq y}$ be the $0$-$1$ loss function.
We are interested in binary classifiers found via adversarial training \labelcref{eq:AT_general}, i.e., minimizers of the problem
\begin{align}\label{eq:AT}
	\inf_{A \in \B(\Omega)} \E_{(x,y)\sim\mu}\left[\sup_{\tilde x \in B(x,\eps)\cap\Omega}\ell(\one_A(\tilde x),y)\right].
\end{align}
We define the conditional distributions $\rho_i := \mu(\cdot\times\{i\})$ for $i=0,1$ and $\rho := \rho_0+\rho_1$.
We pose the following assumption which we shall use in the whole paper, without further reference.
\begin{assumption}[The densities]\label{ass:densities}
We assume that $\rho_0$ and $\rho_1$ (and hence also $\rho$) have densities with respect to the $N$-dimensional Lebesgue measure on $\Omega$ which are continuously differentiable functions, i.e., $\rho_i \in C^1(\Omega)$.
For notational convenience we shall identify $\rho$ and $\rho_i$ with their densities, meaning $\int f\de\rho_{(i)}=\int f\rho_{(i)}\de x$.
Furthermore, we assume that $c_\rho<\rho<c_\rho^{-1}$ in $\Omega$ for some constant $c_\rho>0$.
\end{assumption}

In this situation it follows from the general results in \cite{bungert2023geometry} that problem \labelcref{eq:AT} is equivalent to
\begin{align}\label{eq:AT_Per} 
	\inf_{A \in \B(\Omega)} 
	\iint_{\Omega\times\{0,1\}} 
	\abs{\one_A(x) - y} \de\mu(x,y)
	+
	\eps 
	\Per_\eps(A),
\end{align}
where the generalized perimeter functional $\Per_\eps$ is defined as
\begin{align*}%\label{eq:Per}
	\Per_\eps(A)
	:=
	\frac{1}{\eps}
	\left[
	\int_\Omega 
	\left(\esssup_{B(x,\eps)\cap\Omega}\one_A - \one_A(x)\right)
	\de\rho_0(x)
	+
	\int_\Omega 
	\left(\one_A(x)-\essinf_{B(x,\eps)\cap\Omega}\one_A\right)
	\de\rho_1(x)
	\right].
\end{align*}
Note that, in particular, the supremum in \labelcref{eq:AT} can be replaced by essential suprema and infima in~\labelcref{eq:AT_Per}.
Furthermore, it was proved in \cite{bungert2023geometry} that minimizers to both problems \labelcref{eq:AT,eq:AT_Per} exist and that the infimal values coincide.
Studying the limit of the problem with small adversarial budget, the first and third author showed in \cite{bungert2022gamma} that the perimeter functional $\Gamma$-converges as $\eps\to 0$ to a weighted but local perimeter.
In the current setting with smooth densities this local perimeter is given by $\int_{\partial^* A\cap\Omega} \rho \de \mathcal{H}^{N-1}$, where $\mathcal{H}^{N-1}$ is the Hausdorff (surface) measure and $\partial^*A$ is the measure-theoretic reduced boundary of $A$. 
Therefore, for small values of $\eps$ the problem \labelcref{eq:AT_Per} will effectively minimize the energy
$$\frac{1}{\eps}\iint_{\Omega\times\{0,1\}} 
	\abs{\one_A(x) - y} \de\mu(x,y)
	+ 
	\int_{\partial^* A\cap\Omega} \rho \de \mathcal{H}^{N-1},$$ 
which bears a strong resemblance to the Almgren--Taylor--Wang scheme introduced in \cite{almgren1993curvature} for the study of mean curvature flow, with $\eps>0$ acting as the time step size.
Consequently, the minimization problem \labelcref{eq:AT_Per} should roughly be approximated by a mean curvature flow.
As remarked in the introduction, similar conclusions were drawn in \cite{trillos2022adversarial} on short time horizons.

The natural initial condition for the mean curvature flow is any solution of the adversarial training problem \labelcref{eq:AT_Per} with $\eps=0$:
\begin{align}\label{eq:Bayes_problem}
	\inf_{A \in \B(\Omega)} 
	\iint_{\Omega\times\{0,1\}} 
	\abs{\one_A(x) - y} \de\mu(x,y).
\end{align}
Solutions are called Bayes classifiers and since we have
\begin{align*}
	\iint_{\Omega\times\{0,1\}} 
	\abs{\one_A(x) - y} \de\mu(x,y)
	=
	\int_\Omega 
	\one_A \de\rho_0
	+
	\int_\Omega
	1-\one_{A} \de\rho_1
	=
	-\int_\Omega 
	\one_A \de(\rho_1-\rho_0)
	+
	\rho_1(\Omega),
\end{align*}
problem \labelcref{eq:Bayes_problem} is solved by every set $A$ which is the positive part of a Hahn decomposition of the signed measure $\rho_1-\rho_0$. 
For continuous densities $\rho_0,\rho_1$ any set $A$ which is sandwiched as $\{\rho_1>\rho_0\}\subset A \subset \{\rho_1\geq\rho_0\}$ is a Bayes classifier.

\subsection{The minimizing movements scheme}
\label{sec:MMS}

Now we introduce an iterative adversarial training scheme starting from the Bayes classifier which is a slight modification of \labelcref{eq:AT_Per} and has a rigorous connection to mean curvature flow.
Precisely, we replace \labelcref{eq:AT_Per} by a minimizing movements scheme in the spirit of \cite{almgren1993curvature,luckhaus1995implicit}:
\begin{align}\label{eq:minimizing_movement}
	\begin{dcases}
	A_0 
	&\in 
	\argmin_{A \in \B(\Omega)} 
	\iint_{\Omega\times\{0,1\}} 
	\abs{\one_A(x) - y} \de\mu(x,y),\\
	A_{k+1}
	&\in 
	\argmin_{A \in \B(\Omega)} 
	\int_\Omega 
	\abs{\one_A - \one_{A_k}}
	\frac{\dist(\cdot,\partial A_k)}{\eps}
	\de\rho
	+
	\Per_\eps(A),
	\qquad 
	k\geq 0,
	\end{dcases}
\end{align}
where in this special case we ``overload'' the distance function and define $$\operatorname{dist}(\cdot,\partial A) : = \dist(\cdot,A^1) + \dist(\cdot,\Omega\setminus A^1),$$ with $A^1$ being the points in $A$ with Lebesgue density $1$, as the distance to the boundary of $A$ relative to $\Omega$. The representative set $A^1$ ensures that the distance function does not change when $A$ is modified by a Lebesgue null-set, and we further note that the function coincides with the distance to $\partial (A^1)\cap \Omega$ when $\Omega$ is convex.

We note that this departs from the original adversarial training problem derived in \labelcref{eq:AT_Per} by the inclusion of a distance function. 
At a technical level, this is essential to recover the correct surface velocity for the boundary of the regularized classifier. 
Furthermore, one can show as in \cite[Theorem 5.6]{chan2005aspects} that, if $A_0$ is a smooth set and $\eps>0$ is small, the scheme \labelcref{eq:minimizing_movement} \textit{without} the distance function would stagnate, i.e., $A_k=A_0$ for all $k\in\N$.
At the level of the application, we motivate this term in the following remark.
\begin{remark}[The distance function]
In the context of training a stable classifier the term $\tfrac{\dist(\cdot,\partial A_k)}{\eps}$ acts as an adaptive regularization parameter: 
For points far away from the decision boundary $\partial A_k$ of the previous classifier, the perimeter regularization is unimportant and the first term in \labelcref{eq:minimizing_movement} gets more weight. 
Close to the boundary, the opposite holds true.
If one just performs two iterations of \labelcref{eq:minimizing_movement}, the first $A_0$ equals a Bayes classifier and the second one $A_1$ a solution to adversarial training, where the class labels are distributed according to the Bayes classifier and weighted according to their distance to the respective other class.
Computing this distance function in practice can be done with several different methods, for instance with the fast marching algorithm \cite{sethian1996fast} or the heat flow \cite{crane2013geodesics} based on Varadhans formula \cite{varadhan1967behavior}.
In the high-dimensional settings that are characteristic for machine learning problems, such methods are expensive which is why one resorts to so-called fast minimum norm attacks \cite{pinto2021fast} which computes an approximation of the radius of the smallest ball around a data point which contains an adversarial attack. 
For binary classifiers as in \labelcref{eq:minimizing_movement} this is precisely the distance function to the decision boundary.
\end{remark}
As the solutions of \labelcref{eq:minimizing_movement} are not necessarily unique, we consider a selection procedure following Chambolle's approach in \cite{chambolle2004algorithm}.
To this end, let us introduce the signed distance function of a set $A$ relative to $\Omega$ as
\begin{align}\label{eqn:sdist}
	%\sdist(\cdot,A) := \dist(\cdot,A) - \dist(\cdot,\R^N\setminus A). \\
	\sdist(\cdot,A) := \dist(\cdot,A^1) - \dist(\cdot,\Omega\setminus A^1),
\end{align}
where as before $A^1$ denotes the points in $A$ with Lebesgue density $1$.
Furthermore, we introduce the total variation of a measurable function $u:\Omega\to\R$ which is naturally associated with $\Per_\eps$:
\begin{align}\label{eq:TV}
	\TV_\eps(u)
	:=
	\frac{1}{\eps}
	\left[
	\int_\Omega 
	\left(\esssup_{B(x,\eps)\cap\Omega} u - u(x)\right)
	\de\rho_0(x)
	+
	\int_\Omega 
	\left(u(x)-\essinf_{B(x,\eps)\cap\Omega}u\right)
	\de\rho_1(x)
	\right].
\end{align}
By definition it holds $\Per_\eps(A)=\TV_\eps(\one_A)$ and furthermore the perimeter and the total variation are connected through a coarea formula, see \cite{bungert2023geometry} and \cref{lem:coarea} below.
The central object of study in this paper is the following modified adversarial training scheme
\begin{empheq}[box=\mybluebox]{align}\label{eq:selection}
	\begin{dcases}
	A_0 
	&\in 
	\argmin_{A \in \B(\Omega)} 
	\iint_{\Omega\times\{0,1\}} 
	\abs{\one_A(x) - y} \de\mu(x,y),\\
	w^\ast
	&:=
	\argmin_{w\in L^2(\Omega)} 
	\frac{1}{2\eps}
	\int_\Omega 
	\abs{w - \sdist(\cdot,A_k^c)}^2
	\de\rho
	+
	\TV_\eps(w),
	\qquad 
	k\geq 0,
	\\
	A_{k+1} 
	&:= 
	\{w^\ast > 0\},
	\qquad 
	k\geq 0.
	\end{dcases}
\end{empheq}
We will prove that \labelcref{eq:selection} constitutes a selection mechanism for \labelcref{eq:minimizing_movement}; that is the sequence of sets $(A_k)_{k\in \N_0}$ found via \labelcref{eq:selection} satisfies \labelcref{eq:minimizing_movement}. 
We note that, in contrast to the work of Chambolle \cite{chambolle2004algorithm}, who in our notation considered the scheme $A_{k+1}:=\{w^\ast \leq 0\}$ where $w^\ast:=\argmin_{w\in L^2(\Omega)}\frac{1}{2\eps}\int_\Omega\abs{w - \sdist(\cdot,A_k)}^2\de x+\TV(w)$ and $\TV$ is the standard total variation, we have to flip the sign of the signed distance function and pick the superlevel instead of sublevel set of the resulting minimizer $w^*$ in order for \labelcref{eq:selection} to select a solution of \labelcref{eq:minimizing_movement}. 
This is necessary since $\TV_\eps$ sees orientation, in the sense that ${\TV_\eps(-u)\neq\TV_\eps(u)}$, in contrast to the standard total variation, for which $\TV(-u)=\TV(u)$.

The objective of this paper is to show that, as the adversarial budget vanishes, meaning $\eps\to 0$, the sequence of sets given by \labelcref{eq:selection} converge to a time-parametrized curve $t\mapsto A(t)$ which is a solution of a weighted mean curvature flow equation with the following normal velocity (in the direction $\nu_{A(t)}$)
\begin{empheq}[box=\mybluebox]{align}\label{eq:normal_velocity}
	V(t) = -\frac{1}{\rho}\div\left(\rho \nu_{A(t)}\right)
    = H_{A(t)} - \nabla\log\rho\cdot\nu_{A(t)}
 \quad \text{ on }\partial A.
\end{empheq}
Here $\nu_{A(t)}$ is (a smooth unit-length extension of) the inner unit normal to $\partial A(t)$ and $H_{A(t)}:=-\div\nu_{A(t)}$ denotes the mean curvature of $\partial A(t)$. Note our orientation is such that $H_{A}>0$ if $A$ is a ball. 
The convergence to this mean curvature flow is the content of \cref{thm:main}. 
The mathematical challenges arising in this problem are mostly consequences of the nonlocal $\TV_\eps$ in \labelcref{eq:selection}:  
First, as the $\TV_\eps$ functional is neither local nor smooth, we will need to carefully study its subdifferential and consistency with the $1$-Laplace operator, i.e., the subdifferential of the classical total variation functional. Beyond this, we have not been able to show that minimizers $w^*$ from \labelcref{eq:selection} inherit the regularity of their data, e.g., $\Lip(w^\ast) \leq \Lip(\sdist(\cdot,A^c)) = 1$, an extremely convenient property to have at hand. 
Circumnavigating this obstacle, we instead prove that minimizers are ``almost'' Lipschitz by explicitly constructing sub- and supersolutions for conical data. 
%\todo{check -- While this helps us with proving consistency, it doesn't help us with viscosity solutions since there one needs open or closed evolving sets} 
Finally, in \labelcref{eq:selection}, the parameter $\eps$ (appearing in $\tfrac{1}{2\eps}$ \textbf{and} in $\Per_\eps$) effectively behaves as the time-step in the discretization of a time interval $(0,T)$ and as a non-locality parameter. 
Consequently, the non-locality and time-step are of the same magnitude, and we must ensure that this does not prevent localization of the minimizing movements scheme in the limit as $\eps \to 0$.

\subsection{Main result}
\label{sec:main}

The main consequence of our results is that if the initial Bayes classifier is smooth and compactly contained in $\Omega$, then a time parametrized version of the scheme $(A_k)_{k\in \N_0}$ given in \labelcref{eq:selection} converges to a solution of mean curvature flow with initial condition $A_0$. 
Precisely, we parametrize the sets $(A_k)_{k\in \N_0}$ in \labelcref{eq:selection} with a piecewise-constant curve $t \mapsto A_\eps(t)$ defined by
\begin{equation}\label{eqn:min_mov_param}
A_\eps(t) : = 
A_k \text{ for } t\in [k\eps, (k+1)\eps).
\end{equation}
With this at hand, we may state our result.

\begin{theorem}[Main theorem]\label{thm:main}
Let $\Omega\subset\R^N$ be a bounded and convex domain.
Suppose that in \labelcref{eq:selection} the Bayes classifier $A_0 \subset\subset \Omega$  has $C^2$-boundary and that $t\mapsto A(t)$ is a parameterized curve evolving by the weighted mean curvature flow with normal velocity \labelcref{eq:normal_velocity} up to the first singular time $T_*$, with initial condition $A_0$.

Then as $\eps \to 0$, the time parametrized curves $t \mapsto A_\eps(t)$ defined in \labelcref{eqn:min_mov_param}, coming from the adversarial training scheme \labelcref{eq:selection}, converge in $L^\infty_{\rm loc}([0,T_*);L^1(\Omega;\{0,1\}))$ and in the Hausdorff distance to the weighted mean curvature flow parametrized by $t\mapsto A(t)$.
\end{theorem}
A couple of remarks on this theorem are in order.
\begin{remark}[Smooth Bayes classifiers]
    Note that existence of Bayes classifiers with $C^2$-bound\-ary is guaranteed, e.g., if the levelset $\{\rho_0=\rho_1\}$ is a $C^2$-hypersurface in $\R^N$. 
    This follows from the implicit function theorem if $\rho_0,\rho_1$ are $C^2$-regular in a neighborhood of $\{\rho_0=\rho_1\}$ and if $\nabla\rho_1-\nabla\rho_0\neq 0$ on $\{\rho_0=\rho_1\}$.
\end{remark}
\begin{remark}[Convexity]
    Convexity of the domain $\Omega$ is exclusively used in \cref{lem:cone}, a certain comparison principle for \labelcref{eq:selection} when $A_0$ is a ball. 
    In particular, we believe that the assumption could be avoided with some more work.
\end{remark}
\begin{remark}[The first singular time]
    We also remark that the first singular time in \cref{thm:main} could, for instance, be due to vanishing bubbles, pinch-off, or intersection with $\partial \Omega$.
\end{remark}
\begin{remark}[Generalized solutions of mean curvature flow]
\cref{thm:main} is a direct consequence of \cref{thm:monotone and consistent} further down which states \emph{monotonicity} of the scheme \labelcref{eq:selection} and \emph{consistency} with smooth sub- and superflows (see \cref{def:superflow} below). 
In \cite[Theorem 4]{chambolle2004algorithm} for the Almgren--Taylor--Wang scheme, sub- and superflows are used to define generalized flows that start from more generic initial data so long as the viscosity solution is unique (see also \cite{bellettini1997minimal,chambolle2007approximation,novaga2008implicit} for more general results of that kind). 
A key element for this to work is that the scheme selects a sequence of open (or closed) sets.
However, since we do not have a proof for (Lipschitz) continuity of $w^*$ in \labelcref{eq:selection} (see \cref{cor:almost-Lip} and the discussion preceding it), the iterates $\{w^*>0\}$ of our scheme are in general neither open nor closed.
Alternatively, density estimates can be used to construct open (or closed) selections as in \cite{chambolle2023minimizing}, but in our case those are not available because of the non-locality of~$\TV_\eps$.
As a consequence, it is not clear how to use \eqref{eq:selection} to construct viscosity solutions of the weighted mean curvature flow. 
\end{remark}
\begin{remark}[Boundary conditions]
    Herein, we do not address boundary conditions, but we note that---following the numerical experiments in \cite{etoGiga1}---incorporation of Neumann boundary conditions for the Almgren--Taylor--Wang scheme has only recently been rigorously addressed in \cite{etoGiga2}. 
    Their techniques appear highly PDE dependent, and it is not clear a similar approach can be used in our nonlocal setting.
\end{remark}

\section{Properties of the total variation}\label{sec:totalVarProp}

First, we recall that the total variation admits a coarea formula with respect to the nonlocal perimeter~$\Per_\eps$.
\begin{lemma}[Coarea formula \cite{bungert2023geometry}]\label{lem:coarea}
    For every $u\in L^1(\Omega)$ it holds that
    \begin{align}
        \TV_\eps(u) = \int_\R \Per_\eps(\{u>t\})\de t,
    \end{align}
    where both sides can take the value $+\infty$.
\end{lemma}

We remark that the above lemma is stated in \cite[Proposition 3.13]{bungert2023geometry} using the sets $\{u\geq t\}$.
However, as noted in \cite[Section 4.1]{bungert2022gamma} for sufficiently regular densities (in particular, continuous densities) the statement holds for strict super-level sets, as well. 
%\begin{proof}\ks{should we write ``proof" or should we just cite.}
%    The proof follows from the layer cake representation, see \cite[Proposition 3.13]{bungert2023geometry} and \cite[Section 4.1]{bungert2022gamma}.
%\end{proof}

Next we study some basic properties of the subdifferential of the total variation, regarded as a convex functional on $L^2(\Omega)$ with the standard inner product.
We first record the following lemma which will be familiar to readers used to working with $1$-homogeneous functionals.
\begin{lemma}\label{lem:subdiff}
    Let $\X$ be a Banach space with dual pairing $\langle\cdot,\cdot\rangle:\X^*\times\X\to\R$, and let $J : \X\to[0,\infty]$ be a proper functional with $J(cu)=cJ(u)$ for all $u\in\operatorname{dom} J$ and $c \geq 0$.
    Then the subdifferential of $J$ at $u\in\operatorname{dom}J$, defined as 
    \begin{align}\label{eq:def_subdiff}
        \partial J(u) := 
        \left\lbrace
        p\in\X^*
        \st 
        J(u)
        +
        \langle p, v-u\rangle
        \leq 
        J(v)\; \text{ for all } \; v\in \X
    \right\rbrace,
    \end{align}
    has the characterization
    \begin{align}\label{eq:subdiff_char}
        \partial J(u) 
        =
        \left\lbrace
        p\in \X^* \st 
        \langle p,u\rangle = J(u),\;
        \langle p,v\rangle \leq J(v)\; \text{ for all } \; v\in \X
        \right\rbrace.
    \end{align}
\end{lemma}
\begin{remark}
    Elements $p\in\partial J(u)$ are called subgradients of $J$ at $u$.
\end{remark}
\begin{proof}
The inclusion ``$\supset$'' in \labelcref{eq:subdiff_char} is trivial.
For the converse inclusion, we let $p\in\partial J(u)$ and choose $v=2u$ in \labelcref{eq:def_subdiff}, yielding $J(u) + \langle p,u\rangle \leq J(2u) = 2J(u)$ and hence $\langle p,u\rangle\leq J(u)$.
Choosing $v=0$ and using $J(0)=0$ yields the converse inequality $J(u) \leq \langle p,u\rangle$.
Hence, it holds $\langle p,u\rangle = J(u)$ which immediately also implies $\langle p,v\rangle \leq J(v)$ for all $v\in \X$, using again \labelcref{eq:def_subdiff}.
This concludes the proof of ``$\subset$''.
\end{proof}

It will be important to understand properties of the subdifferential of the total variation $\TV_\eps$, regarded as an extended-valued functional on $L^2(\Omega)$. 
According to \labelcref{eq:def_subdiff} its subdifferential is given by
\begin{align}\label{eq:subdiff_TV}
    \partial\TV_\eps(u) = \left\{p \in L^2(\Omega)\st \TV_\eps(u) + \int_\Omega p(v-u)\de x \leq \TV_\eps(v)\quad\forall v \in L^2(\Omega)\right\}.
\end{align}
Using the characterization \labelcref{eq:subdiff_char} of $\partial\TV_\eps(u)$ with $v \equiv \pm 1$, we note that for $p \in \partial\TV_\eps(u)$ one has $\int_\Omega p \de x= 0$.
Characterizing the subdifferential in full generality beyond \labelcref{eq:subdiff_char} is both not necessary for our purposes and beyond the scope of this paper, for which it suffices to restrict ourselves to suitably nice functions $u$ and a smaller class of test functions than $v \in L^2(\Omega)$.
For this we start with a few informal considerations. 
Since $\TV_\eps(u)$ is positively homogeneous, according to \labelcref{eq:subdiff_char} it suffices to find $p$ such that $\int_\Omega p v \de x \leq \TV_\eps(v)$ for all test functions $v$ with equality for $v=u$ to characterize the subdifferential.
If we assumed that $u$ was sufficiently nice such that $\esssup_{B(x,\eps)\cap\Omega}u$ and $\essinf_{B(x,\eps)\cap\Omega}$ were attained at unique points $\Gamma_\eps(x)$ and $\gamma_\eps(x)$, respectively, we could use a change of variables to obtain
\begin{align*}
    \int_\Omega v \de(\Gamma_\eps)_\sharp\rho_0 \leq \int_\Omega \esssup_{B(\cdot,\eps)\cap\Omega}v\de\rho_0
    \qquad
    \text{and}
    \qquad
    \int_\Omega v \de(\gamma_\eps)_\sharp\rho_1 \geq \int_\Omega \essinf_{B(\cdot,\eps)\cap\Omega}v\de\rho_1
\end{align*}
with equality for $v=u$.
Consequently and not being concerned about regularity, the function 
\begin{align*}
    p := \frac{(\Gamma_\eps)_\sharp\rho_0-\rho_0}{\eps} + \frac{\rho_1-(\gamma_\eps)_\sharp\rho_1}{\eps}    
\end{align*}
would be an element of $\partial\TV_\eps(u)$.
For this to be rigorous, we would have to make sure that the maps $\Gamma_\eps(x) := \argmax_{B(x,\eps)\cap\Omega}u$ and $\gamma_\eps(x) := \argmin_{B(x,\eps)\cap\Omega}u$ are well-defined and the pushforwards $(\Gamma_\eps)_\sharp\rho_0$ and $(\gamma_\eps)_\sharp\rho_1$ have densities in $L^2(\Omega)$.
Towards this goal, we first of all work with sufficiently regular functions $u$ with non-vanishing gradients, and also with a restricted class of test functions $v$ for which we can prove the subdifferential inequality in \labelcref{eq:subdiff_TV} holds.

\begin{proposition}\label{prop:density_of_pushforward}
    Let $u \in C^2(\closure\Omega)$ such that $\abs{\nabla u}\geq c$ in $\closure\Omega$ for a constant $c>0$, and let $\Lambda_{\max}$ denote the largest eigenvalue of the Hessian of $u$ over $\Omega$.
    If $0<\eps < c/\Lambda_{\rm max}$ is small enough, then 
    \begin{itemize}
        \item for every $x\in\Omega_\eps$ the maps 
        \begin{align*}
            \Gamma_\eps(x) := \argmax_{\overline{B(x,\eps)\cap\Omega}} u
            \qquad
            \text{and}
            \qquad
            \gamma_\eps(x) := \argmin_{\overline{B(x,\eps)\cap\Omega}} u
        \end{align*}
        are singletons;
        \item for every $y\in\Omega_{2\eps}$ the densities of the pushforward measures $(\Gamma_\eps)_\sharp\rho_0$ and $(\gamma_\eps)_\sharp\rho_1$ with respect to the Lebesgue measure $\L^N$ are given by \begin{subequations}\label{eq:pushforward_densities}
        \begin{align}
        \frac{\de(\Gamma_\eps)_\sharp\rho_0}
        {\de \L^N}(y)
        &= 
        \frac{\de\rho_0}{\de\L^N}\left(y - \eps\frac{\nabla u(y)}{\abs{\nabla u(y)}}\right)
        \abs{\det\left(\nabla\left(y - \eps\frac{\nabla u(y)}{\abs{\nabla u(y)}}\right)\right)},
        \\
        \frac{\de(\gamma_\eps)_\sharp\rho_1(y)}{\de \L^N}(y)
        &= 
        \frac{\de\rho_1}{\de\L^N}\left(y + \eps\frac{\nabla u(y)}{\abs{\nabla u(y)}}\right)
        \abs{\det\left(\nabla\left(y + \eps\frac{\nabla u(y)}{\abs{\nabla u(y)}}\right)\right)},
    \end{align}
    where by \cref{ass:densities} it holds $\frac{\de\rho_i}{\de\L^N}=\rho_i$ for $i\in\{0,1\}$.
    \end{subequations}
    \item the function $p \in L^2(\Omega_{2\eps})$, defined by
    \begin{align}\label{eq:pushforward_subgradient}
        p := \frac{\de}{\de \L^N}\left[\frac{(\Gamma_\eps)_\sharp\rho_0 - \rho_0}{\eps}
        +
        \frac{\rho_1-(\gamma_\eps)_\sharp\rho_1}{\eps}\right],
    \end{align}      
    satisfies the inequality    \begin{align}\label{eq:subgradient_ineq_pushforward}
        \TV_\eps(u) + \int_\Omega p\varphi\de x \leq \TV_\eps(u+\varphi)
    \end{align}
    for all $\varphi \in L^2(\Omega)$ with $\varphi=0$ almost everywhere in $\Omega\setminus\Omega_{2\eps}$.
    \end{itemize}
\end{proposition}
\begin{proof}
    We will derive the first two statements only for $\Gamma_\eps$; the ones for $\gamma_\eps$ follow from replacing $u$ by $-u$.

 \textit{Step 1 (Optimality condition).}   First, we note that for any $x\in\Omega_{\eps}$ there exists a point $y^* \in \argmax_{\closure{B(x,\eps)\cap\Omega}}u$ since $u$ is continuous.
    Second, by the Karush--Kuhn--Tucker optimality conditions (or direct verification) we get that $y^*$ satisfies
    \begin{align}\label{eq:KKT}
        \nabla u(y^*) - \lambda^*(y^*-x) = 0,
        \qquad
        \abs{y^*-x}\leq \eps,
    \end{align}
    for a Lagrange multiplier $\lambda^*\geq 0$ which is such that $\lambda^*(\abs{y^*-x}^2-\eps^2) = 0$.
    Since by assumption $\abs{\nabla u}\geq c>0$ on $\Omega$, the maximum has to be taken on the boundary of $B(x,\eps)$, i.e., $\abs{y^*-x}=\eps$. 
    Therefore, we obtain from \labelcref{eq:KKT} that the Lagrange multiplier is given by $\lambda^* = \frac{\abs{\nabla u(y^*)}}{\abs{y^*-x}} = \frac{\abs{\nabla u(y^*)}}{\eps}$.

 \textit{Step 2 (Unique maximum).}   Next we prove that the maximizer $y^*$ is uniquely determined. For this, we define the Lagrangian
    \begin{align*}
        L(y,\lambda) := -u(y) + \frac{\lambda}{2}\left(\abs{y-x}^2 - \eps^2\right)\quad \text{ for }\lambda \in [0,\infty)
    \end{align*}
    and observe that it holds
    \begin{align*}
        \nabla^2_yL(y^*,\lambda^*) = -\nabla^2 u(y^*) + \lambda^*\mathbb 1 = -\nabla^2u(y^*) + \frac{\abs{\nabla u(y^*)}}{\eps} \mathbb 1\succeq \left(-\Lambda_{\max} + \frac{c}{\eps}\right)\mathbb 1 \succ 0
    \end{align*}
    by our assumption on $\eps$. We let $M_\eps : = \frac{c}{\eps}-\Lambda_{\max}$ and, supposing that
     $\tilde y \in \partial B(x,\eps)$ is another maximizer, we get using Taylor expansions and applying \labelcref{eq:KKT} that
    \begin{align*}
        -u(\tilde y) 
        &= 
        L(\tilde y,\lambda^*)
        \\
        &=
        L(y^*,\lambda^*) 
        +
        \nabla_y L(y^*,\lambda^*)
        (\tilde y - y^*)
        +
        \frac{1}{2}
        (\tilde y - y^*)^T \nabla_y^2L(y^*,\lambda^*)(\tilde y - y^*)
        +
        o\left(\abs{\tilde y - y^*}^2\right)
        \\
        &\geq 
        -u(y^*) 
        +
        \frac{M_\eps}{2}
        \abs{\tilde y - y^*}^2
        -
        \omega(|\tilde y-y^*|)\abs{\tilde y - y^*}^2
        %o\left(\abs{\tilde y - y^*}^2\right),
    \end{align*}
    where $\omega$ is the modulus of continuity of $\nabla^2_y L( y,\lambda^*)$ in $y$, which is the same as the modulus of continuity of $\nabla^2 u$ (and thereby independent of $\eps$).
    Using that $u(\tilde y) = u(y^*)$, we find
    \begin{align}\label{ineq:contradiction}
        \frac{M_\eps}{2}
        \abs{\tilde y - y^*}^2
        \leq 
        \omega(|\tilde y-y^*|)\abs{\tilde y - y^*}^2.
    \end{align}
    We note that $M_\eps$ is positive for $\eps$ small enough and even $\lim_{\eps\to 0}M_\eps = \infty$.
    Therefore, for $\eps>0$ small enough, \labelcref{ineq:contradiction} becomes a contradiction unless $\tilde y = y^*$.  
    Hence, we have shown that $\Gamma_\eps(x) = \{y^*\}$ is a singleton.

   \textit{Step 3 (Computation of the pushforward).}  
   Vice versa, solving\labelcref{eq:KKT} for $x$, we see that $\Gamma_\eps$ is one-to-one with inverse
    \begin{align}\label{eq:T_inverse}
        \Gamma_\eps^{-1}(y^*) = y^* - \eps\frac{\nabla u(y^*)}{\abs{\nabla u(y^*)}},
    \end{align} and in particular, this is a well-defined injective $C^1$ function on $\Gamma(\Omega_\eps)$. 
    Therefore we can use the definition of the pushforward and the area formula to show for any continuous function $\varphi\in C(\closure\Omega)$ that
    \begin{align*}
        \int_{\Gamma_\eps(\Omega_\eps)} \varphi \de(\Gamma_\eps)_\sharp\rho_0
        &=
        \int_{\Omega_\eps} \varphi\circ \Gamma_\eps\de\rho_0
        \\
        &= \int_{\Gamma_\eps(\Omega_\eps)}\varphi(y)\abs{\det(\nabla\Gamma_\eps^{-1}(y))}\rho_0(\Gamma_\eps^{-1}(y))\de y
        \\
        &= \int_{\Gamma_\eps(\Omega_\eps)}\varphi(y)\abs{\det\left(\nabla\left(y-\eps\frac{\nabla u(y)}{\abs{\nabla u(y)}}\right)\right)}\rho_0\left(y-\eps\frac{\nabla u(y)}{\abs{\nabla u(y)}}\right)\de y.
    \end{align*}
    Restricting ourselves to arbitrary continuous functions $\varphi$ with $\supp\varphi \subset {\Omega_{2\eps}}\subset \Gamma_\eps(\Omega_\eps)$, we obtain the claimed identity \labelcref{eq:pushforward_subgradient} for the pushforwards. 
    Note that ${\Omega_{2\eps}}\subset \Gamma_\eps(\Omega_\eps)$ follows from \labelcref{eq:T_inverse}. 

  \emph{Step 4 (Local subgradient).}  
  To obtain the last claim \labelcref{eq:subgradient_ineq_pushforward}, let $\mathcal{N}$ be the set of points, which are not Lebesgue points for $u+\varphi$ and hence $\mathcal{L}^N(\mathcal{N}) = 0$.
    In particular, also the sets $\mathcal{N}_1 := \Gamma_\eps(\Omega_\eps)\cap \mathcal{N}$ and $\mathcal{N}_2 := \gamma_\eps(\Omega_\eps)\cap \mathcal{N}$ have zero Lebesgue measure.
    Since $\Gamma_\eps^{-1}$ and $\gamma_\eps^{-1}$ are diffeomorphisms from $\Gamma_\eps(\Omega_\eps)$ and $\gamma_\eps(\Omega_\eps)$, respectively, to $\Omega_{\eps}$, we have
    \begin{align*}
        \L^N(\{x\in\Omega_\eps\st\Gamma_\eps(x)\in \mathcal{N}\text{ or } \gamma_\eps(x)\in \mathcal{N}\})
        &=
        \L^N(\{x\in\Omega_\eps\st\Gamma_\eps(x)\in \mathcal{N}_1\text{ or } \gamma_\eps(x)\in \mathcal{N}_2\})
        \\
        &\subset 
        \L^N(\Gamma_\eps^{-1}(\mathcal{N}_1)\cup\gamma_\eps^{-1}(\mathcal{N}_2))
        =0,
    \end{align*}
    hence, for almost every $x\in \Omega_{\eps}$ the points $\Gamma_\eps(x)$ and $\gamma_\eps(x)$ are Lebesgue points of $u+\varphi$. Necessarily, it follows that for such $x\in\Omega_\eps$ it holds
    \begin{align*}
        (u+\varphi)\circ \Gamma_\eps(x) \leq\esssup_{B(x,\eps)}(u+\varphi) \quad \text{ and } \quad \essinf_{B(x,\eps)}(u+\varphi) \leq (u+\varphi)\circ \gamma_\eps(x).
    \end{align*} 
    Hence, we obtain for any $\varphi\in L^2(\Omega)$ with $\varphi=0$ almost everywhere on $\Omega\setminus\Omega_{2\eps}$ that
    \begin{align*}
        \eps\int_\Omega &  p\varphi\de x
        =
        \int_{\Gamma_\eps(\Omega_\eps)}
        \varphi\de(\Gamma_\eps)_\sharp\rho_0
        -
        \int_{\Omega}
        \varphi\de\rho_0
        +
        \int_{\Omega}
        \varphi\de\rho_1
        -
        \int_{\gamma_\eps(\Omega_\eps)}
        \varphi\de(\gamma_\eps)_\sharp\rho_1
        \\
        &=
        \int_{\Omega_\eps}
        \varphi\circ\Gamma_\eps\de\rho_0
        -
        \int_{\Omega}
        \varphi\de\rho_0
        +
        \int_{\Omega}
        \varphi\de\rho_1
        -
        \int_{\Omega_\eps}
        \varphi\circ\gamma_\eps\de\rho_1
        \\
        &=
        \int_{\Omega_\eps}
        \Big((u+\varphi)\circ\Gamma_\eps
       - u\circ\Gamma_\eps\Big)\de\rho_0
        +\int_{\Omega_\eps}\Big(u\circ\gamma_\eps
        -
        (u+\varphi)\circ\gamma_\eps\Big)\de\rho_1
        \\
        &\quad
        -
        \int_{\Omega}
        \varphi\de\rho_0
        +
        \int_{\Omega}
        \varphi\de\rho_1
        \\
        &\leq 
        \int_\Omega 
        \left(\esssup_{B(\cdot,\eps)\cap \Omega}(u+\varphi)
        -
        \esssup_{B(\cdot,\eps)\cap \Omega}u\right)
        \de\rho_0
        -
        \int_\Omega 
        \left(\essinf_{B(\cdot,\eps)\cap \Omega}u
        -
        \essinf_{B(\cdot,\eps)\cap \Omega}(u+\varphi)\right)
        \de\rho_1
        \\
        &\quad
        -
        \int_{\Omega}
        \varphi\de\rho_0
        +
        \int_{\Omega}
        \varphi\de\rho_1
        \\
        &\quad
        -
        \int_{\Omega\setminus\Omega_\eps}
        \left(\esssup_{B(\cdot,\eps)\cap \Omega}(u+\varphi)
        -
        \esssup_{B(\cdot,\eps)\cap \Omega}u\right)
        \de\rho_0
        +
        \int_{\Omega\setminus\Omega_\eps}
        \left(\essinf_{B(\cdot,\eps)\cap \Omega}u
        -
        \essinf_{B(\cdot,\eps)\cap \Omega}(u+\varphi)\right)
        \de\rho_1.
    \end{align*}
    The last two integrals vanish because $\varphi=0$ in $\Omega \setminus \Omega_{2\eps}$. Hence
    \begin{align*}
        \eps\int_\Omega   p\varphi\de x
        &=
        \int_\Omega 
        \left(\esssup_{B(\cdot,\eps)\cap \Omega}(u+\varphi)
        -(u+\varphi)\right)
        \de\rho_0
        +
        \int_\Omega 
        \left(u+\varphi
        -
        \essinf_{B(\cdot,\eps)\cap \Omega}(u+\varphi)\right)
        \de\rho_1
        \\
        &\qquad
        -\int_\Omega 
        \left( \esssup_{B(\cdot,\eps)\cap \Omega}u
        -
        u\right)
        \de\rho_0
        -
        \int_\Omega 
        \left(u
        -
        \essinf_{B(\cdot,\eps)\cap \Omega}u\right)
        \de\rho_1
        \\
        &=
        \eps\TV_\eps(u+\varphi) - \eps\TV_\eps(u).
    \end{align*}
    This shows \labelcref{eq:subgradient_ineq_pushforward} and concludes the proof.
\end{proof}

Next we prove that the subgradient identified in the previous lemma is consistent with a weighted $1$-Laplacian operator which, neglecting boundary conditions, is the subgradient of a local weighted total variation.
\begin{proposition}\label{prop:consistency}
    Under the conditions of \cref{prop:density_of_pushforward} it holds that
    \begin{align}
        \frac{\de}{\de \L^N}\left[\frac{(\Gamma_\eps)_\sharp\rho_0 - \rho_0}{\eps}
        +
        \frac{\rho_1-(\gamma_\eps)_\sharp\rho_1}{\eps}\right]
        =
        -\div\left(\rho\frac{\nabla u}{\abs{\nabla u}}\right) + o_{\eps\to 0}(1)
        \quad
        \text{uniformly in }\Omega_{2\eps}.
    \end{align}
\end{proposition}
\begin{proof}
We start by investigating the density of $(\Gamma_\eps)_\sharp\rho_0$ given in \cref{prop:density_of_pushforward}.
Let us fix $y\in\Omega_{2\eps}$.
Using a Taylor expansion of $\rho_0\in C^1(\closure\Omega)$ and utilizing $\abs{\nabla u(y)}\geq c$, we have uniformly in $\Omega_{2\eps}$ that
\begin{align*}
    \rho_0\left(y-\eps\frac{\nabla u(y)}{\abs{\nabla u(y)}}\right)
    =
    \rho_0(y) - \eps\nabla\rho_0(y)\cdot\frac{\nabla u(y)}{\abs{\nabla u(y)}} + o(\eps).
\end{align*}
Furthermore, using a Taylor expansion of the determinant and utilizing $u \in C^2(\closure\Omega)$ with ${\abs{\nabla u}\geq c}$, we get
\begin{align*}
    \det\left(\nabla\left(y - \eps\frac{\nabla u(y)}{\abs{\nabla u(y)}}\right)\right)
    =
    1 - \eps\div\left(\frac{\nabla u(y)}{\abs{\nabla u(y)}}\right)
    +
    O(\eps^2).
\end{align*}
In particular, we see that the determinant is non-negative if $\eps>0$ is sufficiently small.
Multiplying the two expressions and using the product rule yields
\begin{align*}
    &\phantom{{}={}}
    \rho_0\left(y-\eps\frac{\nabla u(y)}{\abs{\nabla u(y)}}\right)
    \abs{\det\left(\nabla\left(y - \eps\frac{\nabla u(y)}{\abs{\nabla u(y)}}\right)\right)}
    \\
    &=
    \left(\rho_0(y) - \eps\nabla\rho_0(y)\cdot\frac{\nabla u(y)}{\abs{\nabla u(y)}} + o(\eps)\right)
    \left(
    1 - \eps\div\left(\frac{\nabla u(y)}{\abs{\nabla u(y)}}\right)
    +
    O(\eps^2)
    \right)
    \\
    &= 
    \rho_0(y)
    - \eps\nabla\rho_0(y)\cdot\frac{\nabla u(y)}{\abs{\nabla u(y)}} 
    -
    \eps\rho_0(y)
    \div\left(\frac{\nabla u(y)}{\abs{\nabla u(y)}}\right)
    + o(\eps)
    \\
    &=
    \rho_0(y) - \eps\div\left(\rho_0(y)\frac{\nabla u(y)}{\abs{\nabla u(y)}}\right)
    +
    o(\eps).
\end{align*}
Using \cref{prop:density_of_pushforward} we therefore obtain
\begin{align*}
    \frac{\de}{\de \L^N}\left[\frac{(\Gamma_\eps)_\sharp\rho_0 - \rho_0}{\eps}\right]
    =
    -\div\left(\rho_0(y)\frac{\nabla u(y)}{\abs{\nabla u(y)}}\right)
    +
    o_{\eps\to 0}(1).
\end{align*}
Repeating the same arguments for $(\gamma_\eps)_\sharp\rho_1$ and using $\rho=\rho_0+\rho_1$ leads to the final conclusion.
\end{proof}
Next, we prove that the total variation is a lower semicontinuous functional with respect to the weak $L^2$ topology.
In \cite{bungert2023geometry}, it was already proved that it is weak-$*$ lower semicontinuous in $L^\infty$, but this is insufficient for our purposes, because we will use additivity of the subdifferential (see the proof of \cref{prop:select_princi}). 
\begin{lemma}\label{lem:TV_lsc}
    Let $(u_k)_{k\in\N}\subset L^2(\Omega)$ be a sequence which converges weakly to $u\in L^2(\Omega)$.
    Then it holds that
    \begin{align*}
        \TV_\eps(u) \leq \liminf_{k\to\infty} \TV_\eps(u_k).
    \end{align*}
\end{lemma}
\begin{proof}
It suffices to show that $$\TV_\eps^0 (u) : = \frac{1}{\eps}\int_\Omega 
\left(\esssup_{B(x,\eps)\cap\Omega} u - u(x)\right)
\de\rho_0(x)$$ is weakly lower semi-continuous. 
As $\TV^0_\eps$ is convex, it suffices to prove lower semi-continuity with respect to strongly converging sequences. In particular, we assume $u_k \to u$ in $L^2(\Omega)$, and may further suppose that $u_k(x) \to u(x)$ for almost every $x\in \Omega$.

Let $\delta>0$.
The strong convergence of $u_k$ to $u$ in particular implies for all $y\in\Omega$ and $0<r<\delta$:
    \begin{align*} 
        \frac{1}{\L^N(B(y,r)\cap \Omega)}\int_{B(y,r)\cap\Omega}u\de z
        =
        \lim_{k\to\infty}\frac{1}{\L^N(B(y,r)\cap \Omega)}\int_{B(y,r)\cap\Omega}u_k\de z.
    \end{align*}
    Taking Lebesgue points of $u$, for almost every $y\in\Omega$ we get
    \begin{align*}
        u(y) 
        &=
        \lim_{r\to 0}
        \frac{1}{\L^N(B(y,r)\cap \Omega)}\int_{B(y,r)\cap\Omega}u\de z
        =
        \lim_{r\to 0}
        \lim_{k\to\infty}\frac{1}{\L^N(B(y,r)\cap \Omega)}\int_{B(y,r)\cap\Omega}u_k\de z
        \\
        &\leq 
        \liminf_{k\to\infty}
        \esssup_{B(y,\delta)\cap\Omega}u_k.
    \end{align*}
    As $\delta>0$ is arbitrary, we conclude that for almost every $x\in\Omega$
    \begin{align*}
        \esssup_{B(x,\eps)\cap \Omega}u
        &\leq 
        \liminf_{k\to\infty}
        \esssup_{B(x,\eps)\cap\Omega}u_k.
    \end{align*}
    As we have the pointwise convergence of $u_k,$ we improve this to 
    \begin{align*}
        \esssup_{B(x,\eps)\cap \Omega}u - u(x)
        &\leq 
        \liminf_{k\to\infty}\left(
        \esssup_{B(x,\eps)\cap\Omega}u_k - u_k(x) \right),
    \end{align*}
    from which Fatou's lemma concludes the claimed lower semi-continuity.
\end{proof}

Next we establish an important submodularity property of the total variation $\TV_\eps$.
In fact it follows from submodularity of the perimeter $\Per_\eps$ (proved in \cite[Proposition 3.3]{bungert2023geometry}) and the coarea formula (see \cite[Proposition 3.2]{chambolle2010continuous}) but for self-containedness we elaborate on the proof.
\begin{lemma}\label{lem:submodularity}
It holds for all $u,v\in L^2(\Omega)$ that
\begin{align*}
	\TV_\eps(u\vee v)+	\TV_\eps(u\wedge v)
	\leq 
	\TV_\eps(u) + \TV_\eps(v).
\end{align*}
\end{lemma}

\begin{proof}
The statement follows from the directly obtained inequalities 
\begin{align*}
\esssup_{B(\cdot,\eps)\cap\Omega} (u\vee v) &\leq \Big(\esssup_{B(\cdot,\eps)\cap\Omega} u\Big) \vee \Big(\esssup_{B(\cdot,\eps)\cap\Omega} v\Big),\\ 
\esssup_{B(\cdot,\eps)\cap\Omega} (u\wedge v) &\leq \Big(\esssup_{B(\cdot,\eps)\cap\Omega} u\Big) \wedge \Big(\esssup_{B(\cdot,\eps)\cap\Omega} v\Big),
\end{align*} 
the reverse analogues for the $\essinf$, and the elementary identity $a\vee b + a \wedge b = a+b $ for numbers $a,b\in\R$. 
%applied to each of the pairs $ (a,b) = (\esssup u, \esssup v), \, (\essinf u, \essinf v), \, (u,v)$. 
%\todo{left off domains for essinf/esssup... I think its ok here}
\end{proof}

\section{Convergence of the adversarial training scheme}\label{sec:convergence}

For a set $A\in\B(\Omega)$, let us define the \emph{one-step operator} $S_\eps(A)$ of the adversarial training scheme \labelcref{eq:selection} via
\begin{align}\label{eq:selection_operator}
	S_\eps(A)
	:= 
	\{w_\eps^\ast > 0\}
	\quad 
	\text{where}
	\quad
	w_\eps^\ast:=
	\argmin_{u\in L^2(\Omega)} 
	\frac{1}{2\eps}
	\int_\Omega 
	\abs{u - \sdist(\cdot,A^c)}^2
	\de\rho
	+
	\TV_\eps(u).
\end{align}
For convenience, we assume that $w^*_\eps$ is a Lebesgue representative with ${\{w^*_\eps > 0\} = (\{w^*_\eps > 0\})^1}$, where we recall the notation introduced in \labelcref{eqn:sdist}. In particular, with this representative convention in place, one can verify by hand that 
$$\sdist(\cdot,\{w^*_\eps > 0\}^c) = \dist(\cdot, \{w^*_\eps > 0\}^c) - \dist(\cdot, \{w^*_\eps > 0\}), $$
circumventing the need for a well-chosen representative in $\sdist$ (see \labelcref{eqn:sdist}) for another application of $S_\eps$.

In \cref{prop:select_princi} below, we will first prove that the operator $S_\eps$ does in fact select a solution of \labelcref{eq:minimizing_movement}; in other words, the convex minimization problem arising in \labelcref{eq:selection} is consistent with the scheme \labelcref{eq:minimizing_movement}. Second, to prove \cref{thm:main}, we will connect the operator \labelcref{eq:selection_operator} to the limiting equation by showing that the operator is {monotone} and {consistent} with respect to weighted mean curvature flow in the following sense.
\begin{definition}[Monotonicity]\label{def:monotone}
The operator $S_\eps$ defined in \labelcref{eq:selection_operator} is \emph{monotone} if
 ${A' \subset A \subset \Omega}$ implies $S_\eps(A')\subset S_\eps(A)$.
\end{definition}
While monotonicity is a property of the operator by itself, consistency, on the other hand, directly connects the scheme to mean curvature flow. 

For consistency, we rely on the notions of sub- and superflows typically used to construct barrier solutions for mean curvature flow, see for example~\cite[Chapter 9]{bellettini_book}. 
If $[t_0,t_1]\ni t\mapsto A(t)\subset\subset \Omega$ is a smooth curve of smooth sets which evolve with normal speed $V(t) = -\tfrac{1}{\rho}\div\left(\rho \nu_{A(t)}\right)$, where $\nu_{A(t)}$ is the inner normal vector to the boundary of $A(t)$, i.e., as in \labelcref{eq:normal_velocity}, then the signed distance function $d(x,t) := \sdist(x,A^c(t))$ satisfies
\begin{align}\label{eq:distance_fct_PDE}
    \partial_t d(x,t) = \frac{1}{\rho(x)}\div\left(\rho(x)\grad d(x,t)\right)
\end{align}
for any $(x,t)$ with $d(x,t)=0$.
This is because $\nu_{A(t)} = \grad d(x,t)$.
The PDE \labelcref{eq:distance_fct_PDE} forms the basis of our sub- and superflow definitions for weighted mean curvature flow, adapted from \cite[Definition 2.1]{chambolle2007approximation}.
Informally, a superflow is a smooth evolution of sets that moves strictly faster than mean curvature flow, while a subflow moves slower. We emphasize that our meaning is the same as in other works, but the inequalities are reversed as the gradient of the signed distance function $\sdist(x,A^c)$ points into the evolving set (more consistent with $BV$-solution concepts).
%%%%%%%%%%%%%%%%%%%%%%%%%%%%%%%%%%%%%%%%%%%
%%%%%%%%%%%%%%%%%%%%%%%%%%%%%%%%%%%%%%%%%%%%%
\iffalse
\begin{definition}[Super- and subflows]\label{def:superflow}
Let $A(t)\subset\subset \Omega$, $t \in [t_0, t_1]$.
We say that $A(t)$ is a superflow of \labelcref{eq:distance_fct_PDE} if 
\begin{itemize}
    \item there exists a bounded open set $B \subset \Omega$ with $\bigcup_{t_0\leq t\leq t_1}\partial A(t) \times \{t\} \subset B \times [t_0, t_1]$;
    \item $d(x, t) := \sdist(x,A(t)) \in C^1([t_0, t_1]; C^2(B))$;
    \item there exists $\delta>0$ such that 
    \begin{align}\label{ineq:superflow}
        \partial_t d(x,t) \geq \frac{1}{\rho(x)}\div\left(\rho(x)\grad d(x,t)\right) + \delta
    \end{align}
    for any $x \in  B$ and $t \in [t_0, t_1]$. 
\end{itemize}
We say that $A(t)$ is a subflow whenever $\delta < 0$ and
the reverse inequality holds in \labelcref{ineq:superflow}.
\end{definition}
\fi 
%%%%%%%%%%%%%%%%%%%%%%%%%
%%%%%%%%%%%%%%%%%%%%%%%%%%%
\begin{definition}[Sub- and superflows]\label{def:superflow}
Let $A(t)\subset\subset \Omega$, $t \in [t_0, t_1]$.
We say that $A(t)$ is a \emph{subflow} of \labelcref{eq:distance_fct_PDE} if 
\begin{itemize}
   % \item there exists a bounded open set $B \subset \Omega$ with $\bigcup_{t_0\leq t\leq t_1}\partial A(t) \times \{t\} \subset B \times [t_0, t_1]$;
    \item there exists a relatively open set $B \subset \Omega\times [t_0,t_1]$ with $\bigcup_{t_0\leq t\leq t_1}\partial A(t) \times \{t\} \subset B $;

    \item the function $d(x, t) := \sdist(x,A^c(t))$ is continuously differentiable in time and twice continuously differentiable in space in $B$, which we abbreviate as $d\in C^{2,1}_{x,t}(B)$;
    \item there exists $\delta>0$ such that 
    \begin{align}\label{ineq:superflow}
        \partial_t d(x,t) \geq \frac{1}{\rho(x)}\div\left(\rho(x)\grad d(x,t)\right) + \delta
    \end{align}
    for any $(x,t) \in  B$. 
\end{itemize}
We say that $A(t)$ is a \emph{superflow} whenever $\delta < 0$ and
the reverse inequality holds in \labelcref{ineq:superflow}.
\end{definition}
\begin{definition}[Consistency]\label{def:consist}
    The operator $S_\eps$ defined in \labelcref{eq:selection_operator} is \emph{consistent} if
    \begin{itemize}
        \item for every subflow $[t_0,t_1]\ni t\mapsto A(t)$ in the sense of \cref{def:superflow} there exists $\eps_0>0$ such that $S_\eps(A(t))\subset A(t+\eps)$ for all $0<\eps<\eps_0$ and all $t\in[t_0,t_1-\eps]$;
        \item for every superflow the same holds with the converse inclusion.
    \end{itemize}
\end{definition}
The interpretation of this definition is that for $\eps>0$ sufficiently small the scheme $S_\eps(A)$ defined in \labelcref{eq:selection_operator} moves faster than a subflow and slower than a superflow starting at $A$. With these definitions in hand, we may state the principle result of this paper.
\begin{theorem}[Monotonicity and consistency]\label{thm:monotone and consistent}
If $\Omega\subset\R^N$ is a bounded and convex domain the operator $S_\eps$ is monotone and consistent with the weighted mean curvature flow \labelcref{eq:normal_velocity} in the sense of \cref{def:monotone,def:consist}, respectively.
\end{theorem}
\cref{thm:monotone and consistent} will follow directly follow from \cref{prop:monotone,prop:consist} below.
We briefly defer the proofs of the aforementioned selection principle and \cref{thm:monotone and consistent}, as we can now directly conclude \cref{thm:main}.

\begin{proof}[Proof of \cref{thm:main}]
Let $A_0$ and $t\mapsto A(t)$ be as in the hypothesis of the theorem and recall $t\mapsto A_\eps(t)$ defined in \labelcref{eqn:min_mov_param}. Then let $t \mapsto A_{\rm sub}(t)$ be any subflow with initial condition $ A_0 \subset A_{\rm sub}(0)$ and parameter $\delta>0$. 
First, we will show that for any time $t$ before the singular time of $ A_{\rm sub}$, we have that 
\begin{equation}\label{eqn:evoContain}
    A_\eps (t) \subset A_{\rm sub}(\eps \lfloor t/\eps\rfloor).
\end{equation} 
for all $\epsilon$ sufficiently small. 
Similarly, the converse set containment holds for a superflow $A_{\rm sup}$ leading to
\begin{equation}\label{eqn:evoContain_2sides}
    A_{\rm sup}(\eps \lfloor t/\eps\rfloor) \subset A_\eps (t) \subset A_{\rm sub}(\eps \lfloor t/\eps\rfloor) \quad \text{ for all $\eps>0$ sufficiently small}.
\end{equation}
It turns out \labelcref{eqn:evoContain} is a simple consequence of \cref{thm:monotone and consistent}. 
Let $(A_k)_{k\in \N_0}$ be the sets in the definition of $A_\eps$ in \labelcref{eqn:min_mov_param} coming from iteratively applying the scheme. By monotonicity and consistency we have 
%$$A_{\rm sup}(\eps)\subset S_\eps(A_{\rm sup}(0)) \subset S_\eps(A_0) = A_1.$$
$$A_1 = S_\eps(A_0) \subset S_\eps(A_{\rm sub}(0)) \subset  A_{\rm sub}(\eps).$$
 Applying $S_\eps$ to both sides once again, using monotonicity and consistency, we have
%$$A_{\rm sup}(\eps 2) \subset S_\eps(A_{\rm sup}(\eps)) \subset S_\eps(A_1) = A_2.$$
$$A_2 = S_\eps(A_1) \subset S_\eps(A_{\rm sub}(\eps)) \subset A_{\rm sub}(2\eps).$$
Iterating and recalling the definition of $A_\eps$ in \labelcref{eqn:min_mov_param}, we conclude \labelcref{eqn:evoContain} and hence also \labelcref{eqn:evoContain_2sides}. 

Consequently, using \labelcref{eqn:evoContain_2sides} and letting $T$ be the earliest singular time of $A_{\rm sub}$ and $A_{\rm sup}$, we estimate for any $s<T$ that
\begin{equation}\label{eqn:sandwich_ineq}
\limsup_{\eps \to 0}\sup_{t\in [0,s]}\mathcal{L}^N(A_{\eps}(t)\triangle A(t))\leq \sup_{t\in [0,s]}\left(\mathcal{L}^N(A_{\rm sub}(t)\triangle A(t))+ \mathcal{L}^N(A_{\rm sup}(t)\triangle A(t)) \right), 
\end{equation}
where we have used continuity of the sub- and superflows to replace $\eps \lfloor t/\eps\rfloor$ with $t$.
Similarly, one can get an estimate for the Hausdorff distance $d_H(A,B):=\sup_{x\in A}\dist(x,B)\vee\sup_{x\in B}\dist(x,A)$.
Using \labelcref{eqn:evoContain_2sides} again we have
\begin{align}
    \nonumber
    \limsup_{\eps \to 0}\sup_{t\in [0,s]}
    d_H(A_{\eps}(t),A(t))
    &\leq 
    \sup_{t\in [0,s]}
    \sup_{y\in A_{\rm sub}(t)}d(x,A(t)) \vee \sup_{x\in A(t)}\dist(x,A_{\rm sup}(t))
    \\
    \label{eqn:sandwich_ineq_hausdorff}    
    &\leq 
    \sup_{t\in [0,s]}
    d_H(A_{\rm sub}(t),A(t))
    +
    d_H(A_{\rm sup}(t),A(t)).
\end{align}
It remains to argue that the right hand side of \labelcref{eqn:sandwich_ineq,eqn:sandwich_ineq_hausdorff} can be made arbitrarily small by approximating $t\mapsto A(t)$ by sub- and superflows.
Briefly, let $d(x,t)  = \sdist(x,A_{\rm sub}^c(t))$. Note that if $\partial_t d = \frac{1}{\rho}\div(\rho \nabla d) + \delta$ on $\partial A_{\rm sub}(t)$, then as a curvature flow this may be written as 
\begin{equation}\label{eqn:forcedFlow}
V_{\rm sub}(t) = -\frac{1}{\rho}\div\left(\rho \nu_{A_{\rm sub}(t)}\right) -\delta
    = H_{A_{\rm sub}(t)} - \nabla\log\rho\cdot\nu_{A_{\rm sub}(t)} - \delta
 \quad \text{ on }\partial A_{\rm sub}(t),
\end{equation} 
following the convention of \labelcref{eq:normal_velocity}. As \labelcref{eqn:forcedFlow} is a perturbation of \labelcref{eq:normal_velocity}, one can show that if \labelcref{eq:normal_velocity} has a strong solutions up to time $T_*$, then for any $T<T^*$, there is $\delta_T>0$ sufficiently small such that the flow \labelcref{eqn:forcedFlow} has a strong solution up to time $T$ for all $|\delta|<\delta_T$.
With this in mind, we see that the right-hand side of \labelcref{eqn:sandwich_ineq,eqn:sandwich_ineq_hausdorff} can be made arbitrarily small by choosing sub- and superflows satisfying the inequalities of \cref{def:superflow} with equality on the interface and then sending $\delta\to 0$ (in the definition of sub-/superflow one must replace $\delta$ by $\delta/2$ to get the neighborhood $B$); at the same time this will allow one to take $T \to T_*$, concluding the theorem. 
\end{proof}

\subsection{Well-posedness}

We first show that the optimization problem in \labelcref{eq:selection_operator} has a unique solution, up to equality on Lebesgue null-sets.
This follows from the following more general statement.
\begin{proposition}\label{prop:existence_and_Lipschitz}
    For any $f\in L^2(\Omega)$ there exists a unique element $w\in L^2(\Omega)$ such that
    \begin{align*}
	\frac{1}{2\eps}
	\int_\Omega 
	\abs{w-f}^2
	\de\rho
	+
	\TV_\eps(w)
        =
        \inf_{u\in L^2(\Omega)} 
	\frac{1}{2\eps}
	\int_\Omega 
	\abs{u - f}^2
	\de\rho
	+
	\TV_\eps(u).
    \end{align*}
    Furthermore, if $f\in L^\infty(\Omega)$ then $w\in L^\infty(\Omega)$ with $\norm{w}_{L^\infty(\Omega)}\leq\norm{f}_{L^\infty(\Omega)}$.
\end{proposition}
\begin{proof}
    We define the functional $E:L^2(\Omega)\to[0,\infty]$ via
    \begin{align*}
        E(u) := \frac{1}{2\eps}
	\int_\Omega 
	\abs{u - f}^2
	\de\rho
	+
	\TV_\eps(u).
    \end{align*}
    Thanks to \cref{lem:subdiff} the functional $u\mapsto\TV_\eps(u)$ is convex and hence $E$ is strictly convex. 
    This implies uniqueness. Existence of the minimizer $w$ is a consequence of lower semi-continuity of the functional (\cref{lem:TV_lsc}) and the direct method. 
    For the claimed $L^\infty$-bound we note that we can replace $w$ by the truncation $\hat w := (-C)\vee w\wedge C$ with $C:=\norm{f}_{L^\infty(\Omega)}$ which satisfies $E(\hat w)\leq E(w)$ (as may be directly checked using \cref{lem:submodularity}) and therefore by uniqueness it holds that $\hat w = w$, and the bounds for $w$ follow.
\end{proof}

\subsection{Selection property}

Next we show that $S_\eps$ selects a solution of \labelcref{eq:minimizing_movement}. Our proof (lightly) deviates from that of \cite[Proposition 2.2]{chambolle2004algorithm} for the standard ${\rm TV}$ functional due to the asymmetry of $\TV_\eps$ with respect to super- and sublevel sets. In particular, we have $\Per_\eps(A) \neq \Per_\eps(A^c)$.

\begin{proposition}[Selection principle]\label{prop:select_princi}
Letting $S_\eps$ be defined as in \labelcref{eq:selection_operator}, it holds that 
\begin{align*}
    S_\eps(A) \in \argmin_{E\in\B(\Omega)}
    \int_\Omega \abs{\one_E-\one_A}\frac{\dist(\cdot,\partial A)}{\eps}\de\rho 
    +
    \Per_\eps(A).
\end{align*}
\end{proposition}

\begin{proof}

We may of course assume $A\neq \emptyset$ and $A\neq\Omega$.
We use the abbreviation ${d(x):=\sdist(x,A^c)}$ and recall that $|d(x)| = \dist(x,\partial A)$ following the notation introduced after \labelcref{eq:minimizing_movement}.
As the $L^2$-fidelity term and the $\TV_\eps$ functional in the minimization problem \labelcref{eq:selection_operator} are both convex and lower semi-continuous, we have that $$ 0 \in \frac{w^*_\eps - d}{\eps}\rho + \partial \TV_\eps (w^*_\eps) .$$
We define $p:= -\frac{w^*_\eps - d}{\eps}\rho \in \partial \TV_\eps (w^*_\eps)$ and $E_s:= \{w^*_\eps > s\}$. 

\textit{Step 1 (Subdifferential of the super-level sets).}
We \textbf{claim} that for almost every $|s| \leq{\rm diam}(\Omega) =: C$, $$p \in \partial \TV_\eps(\one_{E_s}) =: \partial \Per_\eps(E_s).$$ 
We first note that by the $\eps$-coarea formula in \cref{lem:coarea} and that $\|w_\eps^*\|_{L^\infty}\leq C$ by \cref{prop:existence_and_Lipschitz}, we have
$$\TV_\eps (w_\eps^*) = \int_{-C}^C \Per_\eps(E_s) \de s.$$
Similarly, by the layer cake formula and Fubini's theorem
$$\int_\Omega p w_\eps^* \de x = \int_\Omega p(x)\left(\int_{-C}^C \one_{E_s}(x) \de s - C \right) \de x  =\int_{-C}^C\int_\Omega p(x) \one_{E_s}(x)\de x \de s ,$$
where in the last equality we recall that $\int_\Omega p \de  x = 0 $ as noted immediately after \labelcref{eq:subdiff_TV}. 
By \cref{lem:subdiff}, we have $\int_\Omega p w_\eps^* \de x = \TV_\eps (w_\eps^*)$, so that we may combine the above displays to find that
\begin{equation}\label{eqn:fubini+coarea}
\int_{-C}^C \Per_\eps(E_s) \de s =  \int_{-C}^C\int_\Omega p(x) \one_{E_s}(x)\de x \de s.
\end{equation}
Once again by the characterization of the subdifferential in \cref{lem:subdiff}, we have $\int_\Omega p \one_{E_s}\de x \leq \Per_\eps(E_s)$, so that \labelcref{eqn:fubini+coarea} implies 
\begin{equation}\label{eqn:perimBound}
\Per_\eps(E_s) = \int_\Omega p \one_{E_s}\de x
\end{equation}
for almost every $s$ with $|s|\leq C$. 
Applying the characterization \labelcref{eq:subdiff_char} of the subdifferential gives the claim.

\textit{Step 2 (Subdifferential for $\{w_\eps^*>0\}$).}
The claim of Step 1 can be improved to every $s \in [-C,C)$: Fixing such an $s$ and taking a sequence $s_k \downarrow s$ such that \labelcref{eqn:perimBound} holds for each $s_k$, we have that $\one_{E_{s_k}} \to \one_{E_s}$ pointwise and thereby in $L^1$. We pass to the limit in \labelcref{eqn:perimBound} as $k\to \infty$ using lower semi-continuity of $\Per_\eps$, the $L^1$ convergence, and that $p \in \partial \TV_\eps(w_\eps^*)$ (for the last inequality below) to find
$$\Per_\eps(E_s) \leq \int_\Omega p \one_{E_s}\de x \leq  \Per_\eps(E_s) .$$
Thus $p \in \partial \Per_\eps(E_s)$ for any $s\in [-C,C)$, and in particular $p \in \partial \Per_\eps(E_0) = \partial \Per_\eps(\{w_\eps^*>0\}).$

\textit{Step 3 (Conclusion).}
We apply the definition of subdifferential at $E_0$ for any set $E$ to find that
\begin{equation}\label{eqn:perimSubdiff}
\Per_\eps(E_0) + \int_\Omega p \left(\one_E - \one_{E_0}\right) \de x \leq \Per_\eps(E) .
\end{equation} Noting by definition of $p=\frac{d-w_\eps^*}{\eps}\rho$ that 
$$ \int_\Omega p \left(\one_E - \one_{E_0}\right) \de x 
=
\frac{1}{\eps}\int_\Omega d \left(\one_E - \one_{E_0}\right) \de\rho
-
\frac{1}{\eps}
\int_\Omega 
w_\eps^*(\one_E - \one_{E_0})\de\rho
\geq \frac{1}{\eps}\int_\Omega d \left(\one_E - \one_{E_0}\right) \de\rho,$$
where we used that $w_\eps^*(\one_E - \one_{E_0})\leq 0$ almost everywhere, we can rearrange \labelcref{eqn:perimSubdiff} as
$$ \Per_\eps(E_0) + \frac{1}{\eps}\int_\Omega (-d) \one_{E_0} \de \rho \leq \Per_\eps(E) + \frac{1}{\eps}\int_\Omega (-d) \one_E \de\rho.$$
However, one can verify that $\int_\Omega \abs{\one_E-\one_A}\frac{\dist(\cdot,\partial A)}{\eps}\de\rho = \frac{1}{\eps}\int_\Omega(-d)\one_E \de \rho + \frac{1}{\eps}\int_\Omega d \one_A \de \rho,$ so that the previous display is equivalently written 
$$  \Per_\eps(E_0) + \int_\Omega \abs{\one_{E_0}-\one_A}\frac{\dist(\cdot,\partial A)}{\eps}\de\rho \leq \Per_\eps(E) + \int_\Omega \abs{\one_E-\one_A}\frac{\dist(\cdot,\partial A)}{\eps}\de\rho,$$ 
concluding the proposition.
\end{proof}

\subsection{Monotonicity}

For proving monotonicity we start with a simple comparison principle for solutions of the optimization problem in~\labelcref{eq:selection_operator}.
\begin{proposition}[Comparison principle I]\label{prop:comparison_one}
    For $d,d'\in L^\infty(\Omega)$ with $d^\prime\leq d$ almost everywhere in $\Omega$ assume that $w,w^\prime$ satisfy
\begin{align*}
	w^{(\prime)} = \argmin_{u\in L^2(\Omega)} 
	\frac{1}{2\eps}
	\int_\Omega 
	\abs{u - d^{(\prime)}}^2
	\de\rho
	+
	\TV_\eps(u).
\end{align*}
Then it holds $w^\prime\leq w$ almost everywhere in $\Omega$.
\end{proposition}
\begin{proof}
Using optimality of $w$ and $w'$ we have
\begin{align}\label{eq:optimality}
	\begin{split}
	&\phantom{{}={}}
	\frac{1}{2\eps}
	\int_\Omega 
	\left(\abs{w - d}^2
	+
	\abs{w' - d'}^2\right)
	\de\rho
	+
	\TV_\eps(w)
	+
	\TV_\eps(w')
	\\
	&\leq 
	\frac{1}{2\eps}
	\int_\Omega 
	\left(\abs{w\vee w' - d}^2
	+
	\abs{w\wedge w' - d'}^2\right)
	\de\rho
	+
	\TV_\eps(w\vee w')
	+
	\TV_\eps(w\wedge w').
	\end{split}
\end{align}
Using \cref{lem:submodularity} to cancel the total variations we obtain from the above that
\begin{align*}
	\int_\Omega 
	\left(\abs{w - d}^2
	+
	\abs{w' - d'}^2\right)
	\de\rho
	\leq 
	\int_\Omega 
	\left(\abs{w\vee w' - d}^2
	+
	\abs{w\wedge w' - d'}^2\right)
	\de\rho.
\end{align*}
Expanding squares, canceling terms, and reordering this inequality, we reduce to
\begin{align*}
	0
	&\leq 
	\int_\Omega 
	\left(\left(w-(w\vee w')\right)d
	+
	\left(w'-(w\wedge w')\right)d'\right)
	\de\rho	
	\\
	&= 
	\int_{\Omega	\cap\{w'>w\}}
	\left(w'-w\right)
	\left(d'-d\right)
	\de\rho.
\end{align*}
Using that $\rho>c_\rho$ we infer that $\Omega\cap\{w'>w\}\cap\{d'<d\}$ has zero Lebesgue measure.
As in \cite[Lemma 2.1]{chambolle2004algorithm} one can argue that in fact $w'\leq w$ holds almost everywhere.
\end{proof}

With this comparison principle at hand, the proof of monotonicity for \cref{thm:monotone and consistent} is straightforward.

\begin{proposition}[Monotonicity]\label{prop:monotone}
The operator $S_\eps$ defined in \labelcref{eq:selection_operator} is monotone in the sense of \cref{def:monotone}.
\end{proposition}
\begin{proof}
Let $w'$ and $w$ denote the solutions of the problem in \labelcref{eq:selection_operator} for $A'$ and $A$, respectively. Since $A'\subset A$ we have $\sdist(\cdot,(A')^c)\leq\sdist(\cdot,A^c)$, and by \cref{prop:comparison_one}, it follows that $w'\leq w$ outside of a Lebesgue null-set $\mathcal{N}$. This immediately implies 
\begin{align*}
	S_\eps(A')
	=
	\{w'> 0\} = (\{w'>0\}\setminus \mathcal{N})^1
	\subset (\{w>0\}\setminus \mathcal{N})^1 = 
	\{w> 0\}
	=
	S_\eps(A),
\end{align*}
where we have used that $\{w^{(\prime)}>0\} = (\{w^{(\prime)}>0\})^1 $ by choice of representative in \labelcref{eq:selection_operator} and $(\{w^{(\prime)}>0\}\setminus \mathcal{N})^1 = (\{w^{(\prime)}>0\})^1$ holds for any Lebesgue null-set $\mathcal{N}$.
\end{proof}
%Anyways, one has to choose a good representative of the sets in ATW. Otherwise, the (signed) distance function can be arbitrarily bad. Usually one takes the representative (of the set of finite perimeter) that precisely contains the points of positive density. There's a very short discussion of this in Luckhaus-Sturzenhecker after equation (0.6). The construction of this representative can be found in all the basic GMT books. E.g. in Maggi's book, it's contained in Proposition 12.19.}
\subsection{Consistency}

This subsection is devoted to the proof of consistency for \cref{thm:monotone and consistent}. Consistency connects the numerical scheme directly to mean curvature flow, and as such, will require a delicate analysis. Our approach is motivated by Chambolle and Novaga's in \cite{chambolle2007approximation} for an anisotropic but local $\TV$ functional. 

We briefly summarize the strategy of the proof. To show that the evolving set of a subflow stays outside the adversarial scheme, we show that the subflow for mean curvature flow can be modified to construct a subsolution for a static $\TV_\eps$ problem. To compare the modified subsolution to $w_\eps^*$ found using \labelcref{eq:selection_operator}, we will apply the variational comparison principle proven below in \cref{prop:comparison_two} below on a tubular neighborhood of the interface. 
For this to work, we must know that the modified subsolution is greater than $w_\eps^*$ on the boundary of the tubular neighborhood. This information comes from \cref{lem:cone} below. 
In fact, this lemma can be used to show that up to an error $O(\sqrt{\eps})$, the minimizer $w_\eps^*$ is Lipschitz continuous (see \cref{cor:almost-Lip}), and related estimates will allow us to recover the boundary conditions.

We first note that global minimizers give rise to local minimizers, so long as the boundary conditions are frozen on an $\eps$-neighborhood.

\begin{lemma}[Restricted minimizer]\label{lem:restricted_minimizer}
Let $d\in L^\infty(\Omega)$, and let $w\in L^\infty(\Omega)$ solve
\begin{align*}
    w = \argmin 
    \left\lbrace 
    \frac{1}{2\eps}
    \int_\Omega \abs{u-d}^2\de\rho
    +\TV_\eps(u)
    \st 
    u \in L^2(\Omega)
    \right\rbrace.
\end{align*}
Let $\Omega'\subset\Omega$ be an open subset and recall the notation in \labelcref{eqn:inner_parallel}.
Then it also holds that
\begin{align*}
    w = 
    \argmin 
    \left\lbrace 
    \frac{1}{2\eps}
    \int_{\Omega'} \abs{u-d}^2\de\rho
    +\TV_\eps(u;\Omega')
    \st 
    u \in L^2(\Omega'),\;
    u = w \text{ in }\Omega'\setminus \Omega'_\eps
    \right\rbrace,
\end{align*}
where $\TV_\eps(u;\Omega')$ is the total variation as defined in \labelcref{eq:TV} with $\Omega$ replaced by $\Omega'$.
\end{lemma}
\begin{proof}
    Let $u \in L^2(\Omega')$ be a function such that $u=w$ on $\Omega'\setminus \Omega'_\eps$. 
    We extend $u$ to a function in $L^2(\Omega)$ by setting $u:=w$ on $\Omega\setminus \Omega'$.
    Hence, using also the minimization property of $w$ it holds
    \begin{align*}
        \frac{1}{2\eps}&\int_{\Omega'}\abs{w-d}^2\de\rho + \TV_\eps(w;{\Omega'}) 
        - \left(
        \frac{1}{2\eps}\int_{\Omega'}\abs{u-d}^2\de\rho + \TV_\eps(u;{\Omega'})
        \right)
        \\
        &=
        \frac{1}{2\eps}\int_\Omega\abs{w-d}^2\de\rho + \TV_\eps(w;{\Omega'}) 
        - \left(
        \frac{1}{2\eps}\int_\Omega\abs{u-d}^2\de\rho + \TV_\eps(u;{\Omega'})
        \right)
        \\
        &\leq 
        \TV_\eps(u) - \TV_\eps(u;{\Omega'})
        +
        \TV_\eps(w;{\Omega'}) - \TV_\eps(w)
        \\
        &=
        \frac{1}{\eps}
        \int_{\Omega\setminus\Omega'}
        \bigg(\esssup_{B(x,\eps)\cap\Omega}
        u
        -u(x)\bigg)
        \de\rho_0(x)
        +
        \frac{1}{\eps}
        \int_{\Omega\setminus\Omega'}
        \bigg(u(x)
        -
        \essinf_{B(x,\eps)\cap\Omega}
        u\bigg)
        \de\rho_1(x)
        \\
        &\qquad
        -\frac{1}{\eps}
        \int_{\Omega\setminus\Omega'}
        \bigg(\esssup_{B(x,\eps)\cap\Omega}
        w
        -w(x)\bigg)
        \de\rho_0(x)
        -
        \frac{1}{\eps}
        \int_{\Omega\setminus\Omega'}
        \bigg(w(x)
        -
        \essinf_{B(x,\eps)\cap\Omega}
        w\bigg)
        \de\rho_1(x)
    \end{align*}
    Utilizing that $u=w$ on $\Omega\setminus \Omega'_\eps$ one easily sees that all terms in the right hand side of this inequality cancel which renders it equal to zero. 
    Therefore, since $u$ was arbitrary, $w$ is a minimizer as claimed.
\end{proof}

Next, we establish a comparison principle for solutions of the minimization problem in \labelcref{eq:selection_operator}.
Because the subdifferential of $\TV_\eps$ from \cref{lem:subdiff} is not a differential operator, we use variational instead of PDE techniques to prove this comparison principle.
\begin{proposition}[Comparison principle II]\label{prop:comparison_two}
Let $d\in L^\infty(\Omega)$, let $w\in L^\infty(\Omega)$ solve
\begin{align*}
    w = \argmin 
    \left\lbrace 
    \frac{1}{2\eps}
    \int_\Omega \abs{u-d}^2\de\rho
    +\TV_\eps(u)
    \st 
    u \in L^2(\Omega)
    \right\rbrace,
\end{align*}
and assume that $v\in L^\infty({\Omega'})$ satisfies
\begin{align*}
	\frac{1}{2\eps}
	\int_{\Omega'}
	\abs{v - d}^2
	\de\rho
	+
	\TV_\eps(v;{\Omega'})
	\leq 
        \frac{1}{2\eps}
	\int_{\Omega'}
	\abs{v \vee w - d}^2
	\de\rho
	+
	\TV_\eps(v \vee w;{\Omega'})
\end{align*}
for an open subset ${\Omega'}\subset\Omega$. 
If $v \geq w$ almost everywhere on $\Omega'\setminus \Omega'_\eps$ or if $\Omega'=\Omega$, then $v \geq w$ holds almost everywhere in ${\Omega'}$.
\end{proposition}
\begin{proof}
Let us first abbreviate the energy $E(u):=\tfrac{1}{2\eps}\int_{\Omega'}\abs{u - d}^2\de\rho+\TV_\eps(u;{\Omega'})$ for $u\in L^2({\Omega'})$.
\cref{lem:restricted_minimizer} implies that $w$ is a minimizer of $E$ with fixed data $w$ on $\Omega'\setminus \Omega'_\eps$. 
By the assumption that $v\geq w$ on $\Omega'\setminus \Omega'_\eps$, we get that $v\wedge w = w$ on $\Omega'\setminus \Omega'_\eps$ and hence $v\wedge w$ is a feasible competitor for $w$ on $\Omega'$.
In the case $\Omega'=\Omega$ it is trivial that $v\wedge w$ is a competitor for $w$ on $\Omega$.

Assuming that $v\wedge w\neq w$ on a set of positive measure in $\Omega'$ and using the strict convexity of $E$, we get
\begin{align*}
	E(w) < E(v\wedge w).
\end{align*}
Furthermore, by assumption we get that
\begin{align*}
	E(v) \leq E(v\vee w).
\end{align*}
Summing these two inequalities and using \cref{lem:submodularity} to cancel the total variations we get
\begin{align*}
	\int_{\Omega'} \abs{v-d}^2+\abs{w-d}^2\de\rho 
	&<
	\int_{\Omega'} 
	\abs{v\vee w-d}^2+\abs{v\wedge w-d}^2\de\rho
	% \\
	% &=
	% \int_{\Omega'} 
	% (v\vee w)^2
	% -2(v\vee w)d
	% +(v\wedge w)^2
	% +2(v\wedge w)d	
	% \de\rho
	% \\
	% &=
	% \int_{\Omega'} 
	% v^2 + w^2
	% -2v d
	% +2wd	
	% \de\rho
	\\
	&=
	\int_{\Omega'} \abs{v-d}^2+\abs{w-d}^2\de\rho
\end{align*}
which is a contradiction. 
Hence, we have $v\wedge w=w$ almost everywhere on ${\Omega'}$, proving the claim.
\end{proof}

In the next lemma we derive a subsolution of the optimization problem in \labelcref{eq:selection_operator} where the data is a cone which corresponds to controlling the action of the scheme \labelcref{eq:selection} on a ball. 
Note that, as opposed to the case of constant densities and a local total variation, we cannot compute the explicit solution. 
However, a subsolution suffices for our purposes. For consistency with the language in \cref{def:superflow}, a subsolution is minimal with respect to competitors that are greater than it.

\begin{lemma}[Subsolution for cone data]\label{lem:cone}
    Let $\Omega\subset \R^N$ a bounded and convex domain, $x_0 \in \Omega$, and define $d(x) := \abs{x-x_0}$ for $x\in\R^N$. 
    Let furthermore 
    \begin{align*}
        w := \argmin_{u \in L^2(\Omega)}\frac{1}{2\eps}\int_{\Omega}\abs{u-d}^2\de\rho + \TV_\eps(u).
    \end{align*}
    There exist constants $C_1,C_2\geq 1$ with $2C_2\geq C_1> C_2$ and depending only on  $\Lip(\rho_0)$, $\Lip(\rho_1)$, $c_\rho$, $\operatorname{diam}(\Omega)$, and the dimension $N$, such that for $\eps>0$ sufficiently small it holds for almost every $x\in\Omega$ that
    \begin{align*}
        w(x) \leq 
        \overline w(x) :=
        \begin{dcases}
            C_1\sqrt{\eps}\qquad&\text{if } \abs{x-x_0}\leq C_2\sqrt{\eps},\\
            \abs{x-x_0} + \frac{C_2(C_1-C_2)\eps}{\abs{x-x_0}}\qquad&\text{else}.
        \end{dcases}
    \end{align*}
\end{lemma}
\begin{remark}\label{rem:supersolution_cone}
    For a negative cone, i.e., $d(x)=-\abs{x-x_0}$, it is not immediately obvious that $-\overline w$ will be a supersolution, in the sense that $w\geq-\overline w$. 
    The reason for this is that $\TV_\eps(-u)\neq \TV_\eps(u)$.
    However, since all constants in the definition of $\overline w$ only depend on the Lipschitz constants of $\rho_0$ and $\rho_1$, the density lower bound $c_\rho$, and the dimension $N$, one can just exchange the roles of the densities in the definition of $\TV_\eps$ and reduce to the subsolution case of \cref{lem:cone}. 
\end{remark}
\begin{proof}
    For a lighter notation we assume without loss of generality that $x_0=0 \in \Omega$. Throughout the proof, we let $C_1\geq 1$ and $C_2\geq 1$ be the constants from the lemma statement, deferring their specific choice to the last step of the proof.
    The strategy is to construct $p\in L^2(\Omega)$ that satisfies
    \begin{align}
        \label{eq:positive_subgrad}
        (\overline w - d)\rho + \eps p &\geq 0 && \text{in } \Omega,
        \\
        \label{eq:cone_subdifferential}
        \int_\Omega p \varphi \de x &\leq \TV_\eps(\overline w+\varphi) - \TV_\eps(\overline w) && \text{for all }\varphi\in L^2(\Omega),\,\varphi\geq 0.
    \end{align}
    These two properties immediately imply that for the energy $E(u):=\tfrac{1}{2\eps}\int_{\Omega}\abs{u - d}^2\de\rho+\TV_\eps(u)$ for $u\in L^2({\Omega})$ it holds $E(\overline w) \leq E(\varphi + \overline w)$ for all non-negative test functions $\varphi\geq 0$.
    Then \cref{prop:comparison_two} with the choice $\varphi := w\vee\overline w - \overline w\geq 0$ implies the claim that $w\leq\overline w$.
    The required function $p$ is reminiscent of a subgradient of $\TV_\eps$ at $\overline w$ with the difference being that it only satisfies the subdifferential inequality \labelcref{eq:cone_subdifferential} for non-negative test functions. 

    \textit{Step 1 (Construction of a ``subgradient'').}
    Since $\overline w$ is a radial function, constructing $p$ is not difficult. In the spirit of \cref{prop:density_of_pushforward}, we will define argmin and argmax operators corresponding to the capped cone $\overline w$ on $\R^N$.
    We first note that $\overline w$ is radially non-decreasing. 
    Indeed, the derivative of the function $f(r):=r+\frac{C_2(C_1-C_2)\eps}{r}$ satisfies for $r\geq C_2\sqrt{\eps}$:    
    \begin{align*}
        f'(r) = 1 - \frac{C_2(C_1-C_2)\eps}{r^2} \geq 1 - \frac{C_2(C_1-C_2)}{C_2^2}
        =
        1-\frac{C_1-C_2}{C_2} = \frac{2C_2-C_1}{C_2}\geq 0.
    \end{align*}
    To construct $p$ we begin by defining the argmax map
    \begin{align}
        \Gamma_\eps(x) := 
            \sigma_\eps(\abs{x})\frac{x}{\abs{x}}
    \end{align}
    where the piecewise linear and increasing function $\sigma_\eps$ is defined as
    \begin{align*}
        \sigma_\eps(t) : = t+ \eps
        \min\left\lbrace
        \frac{t}{C_2\sqrt{\eps}-\eps},1
        \right\rbrace
        = 
        \left\{\begin{aligned}
            &\frac{t}{1- \sqrt{\eps}/C_2}
            \quad && \text{ if }0\leq t < C_2\sqrt{\eps}-\eps,\\
            &t+\eps\quad && \text{ if }t\geq C_2\sqrt{\eps}-\eps,
        \end{aligned}\right.
    \end{align*}
    which in particular satisfies $\sigma_\eps(0) = 0$ so that $\Gamma_\eps$ is well defined at $x=0$.
    Note that we can assume $\eps< 1$ so that $C_2\sqrt{\eps}-\eps>0$.
    It is important to note that here $\Gamma_\eps$ is a ``global'' argmax map of $\overline w$ that does not see the geometry of the domain $\Omega$, i.e., $\Gamma_\eps(x)\in\argmax_{B(x,\eps)}\overline w$.
    Note that $\Gamma_\eps$ is invertible and by construction an $\eps$-perturbation of the identity. 
    More precisely, the function $\sigma_\eps$ is invertible on $[0,\infty)\to[0,\infty)$ with inverse $\tau_\eps := \sigma_\eps^{-1}$.
    Hence, the inverse of $\Gamma_\eps$ is given by
    \begin{align*}
        \gamma_\eps (x):=\Gamma_\eps^{-1}(x) = \tau_\eps(\abs{x})\frac{x}{\abs{x}}
    \end{align*}
    and thanks to the convexity of $\Omega$ it holds that $\gamma_\eps(\Omega)\subset\Omega$.
    It is immediate from the piecewise definition of $\sigma_\eps$ that
    \begin{align}\label{eq:tau_explicit}
        \tau_\eps(s)
        :=
        s - 
        \eps 
        \min\left\lbrace
        \frac{s}{C_2\sqrt{\eps}},1
        \right\rbrace
        =
        \left\{\begin{aligned}
            &(1-\sqrt{\eps}/C_2) s
            \quad && \text{ if }0\leq s < C_2\sqrt{\eps},\\
            &s-\eps\quad && \text{ if }s\geq C_2\sqrt{\eps}.
        \end{aligned}\right.
    \end{align}
    Using this it is straightforward to see that $\gamma_\eps $ is an argmin map for $\overline w$ on $\R^N$.
    
    We can now define the functions
    \begin{alignat*}{2}
        p_0(x) &:= \frac{1}{\eps}\bigg(\rho_0(\Gamma_\eps^{-1}(x))\abs{\det \nabla\Gamma_\eps^{-1}(x)} - \rho_0(x)\bigg),
        \\
        p_1(x) &:= \frac{1}{\eps}\bigg(\rho_1(x) - \rho_1(\gamma_\eps^{-1}(x))\abs{\det \nabla\gamma_\eps^{-1}(x)}\bigg)\one_{\gamma_\eps(\Omega)},
        \\
        p(x) &:= p_0(x) + p_1(x).
    \end{alignat*}
    Note that the reason why we have to introduce the characteristic function $\one_{\gamma_\eps(\Omega)}$ in the definition of $p_1$ is that the argmax map $\Gamma_\eps=\gamma_\eps^{-1}$ exits $\Omega$, i.e., $\gamma_\eps^{-1}(\Omega)\not\subset\Omega$.
    
    \textit{Step 2 (Validity of the ``subdifferential'' inequality).}
    Next we prove \labelcref{eq:cone_subdifferential} by estimating the $L^2$-inner product of $p_i$ and $\varphi$ for $i=0,1$, where slightly different arguments are required.
    We start with $i=0$.
    Using a change of variables and $\Gamma_\eps(x)\in\argmax_{B(x,\eps)\cap\Omega}\overline w$ for $x\in\Gamma_\eps^{-1}(\Omega)$ we find
    \begin{align*}
        \eps \int_\Omega p_0\varphi\de x
        &=
        \int_\Omega
        \left(\rho_0(\Gamma_\eps^{-1})\abs{\det \nabla\Gamma_\eps^{-1}} - \rho_0\right)\varphi\de x
        \\
        &=
        \int_{\Gamma_\eps^{-1}(\Omega)}
        \varphi \circ \Gamma_\eps \de\rho_0 - \int_\Omega\varphi\de\rho_0
        \\
        &=
        \int_{\Gamma_\eps^{-1}(\Omega)}
        (\overline w +\varphi)\circ \Gamma_\eps \de\rho_0 - \int_\Omega \overline w +\varphi \de\rho_0
        -
        \left[
        \int_{\Gamma_\eps^{-1}(\Omega)}
        \overline w\circ \Gamma_\eps \de\rho_0
        -
        \int_\Omega\overline w\de\rho_0
        \right]
        \\
        &\leq 
        \int_\Omega 
        \left(\esssup_{B(\cdot,\eps)\cap\Omega}(\overline w +\varphi)
        -
        (\overline w +\varphi)\right)
        \de\rho_0
        -
        \int_\Omega 
        \left(\esssup_{B(\cdot,\eps)\cap\Omega}\overline w
        -
        \overline w\right)
        \de\rho_0
        \\
        &\qquad
        +\int_{\Omega\setminus\Gamma_\eps^{-1}(\Omega)} 
        -\esssup_{B(\cdot,\eps)\cap\Omega}(\overline w +\varphi)+
        \esssup_{B(\cdot,\eps)\cap\Omega}\overline w
        \de\rho_0.
    \end{align*}
    The last integral on the right-hand side is non-positive because $\varphi \geq0$. In the last inequality above we used $\overline w\circ\Gamma_\eps(x) = \esssup_{B(x,\eps)\cap \Omega}\overline w$ for $x\in\Gamma_\eps^{-1}(\Omega)$, and that $(\varphi+\overline w)\circ\Gamma_\eps \leq \esssup_{B(\cdot,\eps)\cap \Omega}(\varphi+\overline w)$ almost everywhere in $\Gamma_\eps^{-1}(\Omega)$.
    The reasoning for the latter inequality to hold is analogous to the one in the proof of \cref{prop:density_of_pushforward}, using that $\Gamma_\eps$ is a Lipschitz isomorphism and hence preserves Lebesgue null-sets.
    
    The case $i=1$ is treated similarly, although not entirely symmetrically.
    Using the fact that $\varphi \geq0$ and that by convexity $\gamma_\eps(\Omega)\subset\Omega$ we get 
    \begin{align*}
        \eps \int_\Omega p_1\varphi\de x
        &=
        \int_{\gamma_\eps(\Omega)}
        \left(\rho_1-\rho_1(\gamma_\eps^{-1})\abs{\det \nabla\gamma_\eps^{-1}}\right)\varphi\de x
        \\
        &=
        \int_{\gamma_\eps(\Omega)}
        \varphi
        \de \rho_1
        -
        \int_{\gamma_\eps^{-1}(\gamma_\eps(\Omega))}
        \varphi\circ \gamma_\eps
        \de\rho_1
        \\
        &\leq
        \int_{\Omega}
        \varphi
        \de \rho_1
        -
        \int_{\Omega}
        \varphi\circ \gamma_\eps
        \de\rho_1
        \\
        &=
        \int_{\Omega}
        (\overline w +\varphi)
        \de \rho_1
        -
        \int_{\Omega}
        (\overline w +\varphi)\circ \gamma_\eps
        \de\rho_1
        -\left[
        \int_{\Omega}
        \overline w \de \rho_1
        -
        \int_\Omega 
        \overline w \circ \gamma_\eps \de\rho_1
        \right]
        \\
        &\leq 
        \int_\Omega
        \left((\overline w +\varphi)
        -
        \essinf_{B(\cdot,\eps)\cap\Omega}
        (\overline w +\varphi)\right)
        \de\rho_1
        -
        \int_\Omega
        \left(\overline w
        -
        \essinf_{B(\cdot,\eps)\cap\Omega}
        \overline w\right)
        \de\rho_1.
    \end{align*} 
    Here again, we need to argue as before that $\gamma_\eps$ preserves Lebesgue null-sets (as a diffeomorphism) for the validity of the last inequality.
    Adding the two inequalities we have just established and dividing by $\eps>0$ proves \labelcref{eq:cone_subdifferential}.

 \textit{Step 3 (Optimality conditions for supersolution).}
    It remains to prove \labelcref{eq:positive_subgrad}, i.e., the inequality $(\overline w-d)\rho + \eps p \geq 0$. The basic idea is that $\overline w-d \approx \sqrt{\eps}$ by the choice of $\overline w$, so that the inequality will follow if we can show $\eps p\geq-\sqrt{\eps}$ (up to a constant multiple). Importantly, within this step, we let $\crho >0$ be a constant (possibly changing from line to line) depending only on ${\rm Lip}(\rho_i)$, $c_\rho$, ${\rm diam}(\Omega)$, and the dimension $N.$
    
    For this we first compute $\det \nabla\Gamma_\eps^{-1}(x)$ which appears in the definition of $p_0$. 
    The Jacobian of $\Gamma_\eps^{-1}$ is given by
    \begin{align*}
        \nabla\Gamma_\eps^{-1}(x)
        &=
        \tau_\eps'(\abs{x})\frac{x}{\abs{x}}\otimes\frac{x}{\abs{x}}
        +
        \frac{\tau_\eps(\abs{x})}{\abs{x}}
        \left(
        \mathbb 1
        -
        \frac{x}{\abs{x}}
        \otimes
        \frac{x}{\abs{x}}
        \right)
        \\
        &=
        \frac{\tau_\eps(\abs{x})}{\abs{x}}
        \left[
        \mathbb 1
        +
        \left(
        \tau_\eps'(\abs{x})
        \frac{\abs{x}}{\tau_\eps(\abs{x})}
        -
        1
        \right)
        \frac{x}{\abs{x}}
        \otimes 
        \frac{x}{\abs{x}}       
        \right].
    \end{align*}
    The derivative of $\tau_\eps$ defined in \labelcref{eq:tau_explicit} is given by
    $
        \tau_\eps'(t) 
        =
        1-\frac{\sqrt{\eps}}{C_2}\one_{\{t\leq C_2\sqrt{\eps}\}},
    $
    so that 
     \begin{align*}
        \left(
        \tau_\eps'(\abs{x})
        \frac{\abs{x}}{\tau_\eps(\abs{x})}
        -
        1
        \right) 
        =\begin{cases}
        0 & \text{ if }|x|\leq C_2 \sqrt{\eps}, \\
        \frac{\eps}{|x|-\eps} & \text{ if }|x|\geq C_2\sqrt{\eps}   .   
        \end{cases}
    \end{align*}
    
    Using this, we make a case distinction based on $\abs{x}$:
    For $\abs{x}< C_2\sqrt{\eps}$, using also the elementary inequality $(1+x)^N\geq 1+N x$ for $x\geq -1$, we get
    \begin{align}\label{eq:bound_determinant_inner}
        \det(\nabla\Gamma_\eps^{-1}(x))
        =
        \left(\frac{\tau_\eps(\abs{x})}{\abs{x}}\right)^N
        \det(\mathbb 1)
        =
        \left(1-\frac{\sqrt{\eps}}{C_2}\right)^N
        \geq 
        1-\frac{\sqrt{\eps}N}{C_2}
    \end{align}
    since $\eps\leq 1\leq  C_2^2$.
    For $\abs{x}\geq C_2\sqrt{\eps}$ we can use a Taylor expansion of the determinant to get
    \begin{align}
        \nonumber
        \det(\nabla\Gamma_\eps^{-1}(x))
        &=
        \left(
        1-\frac{\eps}{\abs{x}}
        \right)^N
        \det\left(
        \mathbb 1
        +
        \frac{\eps}{\abs{x}-\eps}
        \frac{x}{\abs{x}}
        \otimes
        \frac{x}{\abs{x}}
        \right)
        \\
        &\geq
        \label{eq:bound_determinant_outer}
        \left(
        1-\frac{\eps N}{\abs{x}}
        \right)
        \left(
        1+
        \frac{\eps}{\abs{x}-\eps}
        +
        O\left(\left(\frac{\eps}{\abs{x}-\eps}\right)^2\right)\right)
        \geq
        \left(
        1-\frac{\eps N}{\abs{x}}
        \right),
    \end{align}        
    where we note that for $\eps>0$ sufficiently small the term $\frac{\eps}{\abs{x}-\eps}$ dominates the quadratic one and, consequently, both can be dropped.

    Using Lipschitz continuity of $\rho_0$ and the explicit formula for $\tau_\eps$ in \labelcref{eq:tau_explicit} we also have
\begin{align}
        \nonumber
        \rho_0(\Gamma_\eps^{-1}(x)) = 
        \rho_0\left(\tau_\eps(\abs{x})\frac{x}{\abs{x}}\right)
        &\geq 
        \rho_0(x) - \Lip(\rho_0)\abs{x-\left(\abs{x}-\eps\min\left\lbrace\frac{\abs{x}}{C_2\sqrt{\eps}},1\right\rbrace\right)\frac{x}{\abs{x}}}
        \\
        &=
        \rho_0(x)-\eps \crho
        \min\left\lbrace\frac{\abs{x}}{C_2\sqrt{\eps}},1 \right\rbrace.   \label{eq:rho_tau}
    \end{align}
    Combining \labelcref{eq:bound_determinant_inner}, \labelcref{eq:bound_determinant_outer,eq:rho_tau} we obtain the following lower bound for $p_0$:
\begin{align}
        \nonumber
        \eps p_0(x) 
        &\geq  
        \left(\rho_0(x)-
        \eps\crho
        \min\left\lbrace\frac{\abs{x}}{C_2\sqrt{\eps}},1 \right\rbrace\right)
        \left(
        1 - \frac{\eps N}{\max\{\abs{x},C_2\sqrt{\eps}\}}
        \right)
        -\rho_0(x)
        \\\nonumber 
        &\geq
        -\sqrt{\eps}\crho\left(\min\left\lbrace\frac{\abs{x}}{C_2},\sqrt{\eps} \right\rbrace
        +\frac{\sqrt{\eps}}{\max\{\abs{x},C_2\sqrt{\eps}\}}\right). 
    \end{align} 
    From this we obtain two lower bounds---a generic one and an improved estimate away from the cone tip: 
    \begin{align}
        \label{eq:LB_p0_generic}
        \eps p_0(x) 
        &\geq -\sqrt{\eps}\crho\left(\sqrt{\eps} + \frac{1}{C_2}\right) \geq -\sqrt{\eps}\crho
         &&\text{for all } x\in\Omega,
        \\
        \label{eq:LB_p0}
        \eps p_0(x) 
        &\geq 
        -\sqrt{\eps} \crho \left(\sqrt{\eps} +\frac{\sqrt{\eps}}{|x|}\right) \geq - \frac{\eps}{|x|}\crho && \text{if }\abs{x}\geq C_2\sqrt{\eps},
    \end{align}
    where in the last inequality we have absorbed $\operatorname{diam}(\Omega)$ into $\crho$.
    
    We continue with proving a similar bound for $p_1$.
    For this we remember that $\gamma_\eps^{-1}(x)=\Gamma_\eps(x) = \sigma_\eps(\abs{x}) \frac{x}{\abs{x}}$.
    Analogous to before, we get that the Jacobian is
    \begin{align*}
        \nabla\gamma_\eps^{-1}(x)
        =
        \frac{\sigma_\eps(\abs{x})}{\abs{x}}
        \left[
        \mathbb 1
        +
        \left(
        \sigma_\eps'(\abs{x})
        \frac{\abs{x}}{\sigma_\eps(\abs{x})}
        -
        1
        \right)
        \frac{x}{\abs{x}}
        \otimes 
        \frac{x}{\abs{x}}       
        \right],
    \end{align*}
     and we find
    $
        \sigma_\eps'(t)
        =
        1+\frac{\sqrt{\eps}}{C_2-\sqrt{\eps}}\one_{\{t\leq C_2\sqrt{\eps}-\eps\}}
    $
    and compute 
    \begin{align*}
        \left(
        \sigma_\eps'(\abs{x})
        \frac{\abs{x}}{\sigma_\eps(\abs{x})}
        -
        1
        \right) 
        =\begin{dcases}
        0 & \text{ if }|x|\leq C_2 \sqrt{\eps}-\eps, \\
        -\frac{\eps}{|x|+\eps} & \text{ if }|x|\geq C_2\sqrt{\eps}-\eps   .   
        \end{dcases}
    \end{align*}  
    Making case distinctions, as before, and also using that $(1+x)^N\leq 1+2Nx$ for sufficiently small $x,$ it holds for $ \abs{x}\leq C_2\sqrt{\eps}-\eps$ that
    \begin{align}\label{eq:bound_determinant_inner_1}
        \det(\nabla\gamma_\eps^{-1}(x))
        =
        \left(
        1 + \frac{\eps}{C_2\sqrt{\eps}-\eps}
        \right)^N
        \leq 
        1 + \frac{2\sqrt{\eps}N}{C_2-\sqrt{\eps}}
    \end{align}
    whenever $\eps>0$ is sufficiently small (depending on $C_2$).
    Similarly, for $\abs{x}\geq C_2\sqrt{\eps}-\eps$, we have
    \begin{align}
        \nonumber
        \det(\nabla\gamma_\eps^{-1}(x))
        &\leq 
        \left(1 + \frac{2\eps N}{\abs{x}}\right)
        \left(
        1
        -
        \frac{\eps}{\abs{x}+\eps}
        +
        O\left(
        \left(
        \frac{\eps}{\abs{x}+\eps}
        \right)^2
        \right)
        \right)
        \\        \label{eq:bound_determinant_outer_1}
        &\leq 
        \left(1 + \frac{2\eps N}{\abs{x}}\right)
    \end{align}
    for $\eps>0$ sufficiently small.
    Using Lipschitz continuity of $\rho_1$ we have
    \begin{align}
        \nonumber
        \rho_1(\gamma_\eps^{-1}(x)) 
        = 
        \rho_1\left(\sigma_\eps(\abs{x})\frac{x}{\abs{x}}\right)
        &\leq 
        \rho_1(x) +
        \Lip(\rho_1)
        \abs{x-\left(\abs{x} + \eps\min\left\lbrace\frac{\abs{x}}{C_2\sqrt{\eps}-\eps},1\right\rbrace\right)\frac{x}{\abs{x}}}
        \\
        \label{eq:rho_tau-1}
        &=
        \rho_1(x)+\eps\crho
        \min\left\lbrace\frac{\abs{x}}{C_2\sqrt{\eps}-\eps},1\right\rbrace.
    \end{align}
    Combining \labelcref{eq:bound_determinant_inner_1}, \labelcref{eq:bound_determinant_outer_1,eq:rho_tau-1} we obtain the following lower bound on $p_1(x)$ for $x\in\gamma_\eps(\Omega)$:
    \begin{align}\nonumber
        \eps p_1(x) 
        &\geq  
        \rho_1(x)
        -
        \left(
        \rho_1(x)+\eps\crho
        \min\left\lbrace\frac{\abs{x}}{C_2\sqrt{\eps}-\eps},1\right\rbrace\right)
        \left(
        1 + \frac{2\eps N}{\max\{\abs{x},C_2\sqrt{\eps}-\eps\}}
        \right)
        \\\nonumber
        &= 
        -\sqrt{\eps}\crho
        \left(\min\left\lbrace \frac{\abs{x}}{C_2-\sqrt{\eps}},\sqrt{\eps}\right\rbrace
        +
        \frac{\sqrt{\eps}}{\max\{\abs{x},C_2\sqrt{\eps}-\eps\}}\right)
        \\\nonumber 
        &\qquad-
        \sqrt{\eps} \crho
        \min\left\lbrace\frac{\abs{x}}{C_2-\sqrt{\eps}},\sqrt{\eps}\right\rbrace
        \frac{\eps}{\max\{\abs{x},C_2\sqrt{\eps}-\eps\}}
        \\\nonumber 
        &\geq 
         -\sqrt{\eps}\crho
        \left(\min\left\lbrace \frac{\abs{x}}{C_2-\sqrt{\eps}},\sqrt{\eps}\right\rbrace
        +
        \frac{\sqrt{\eps}}{\max\{\abs{x},C_2\sqrt{\eps}-\eps\}}\right)
    \end{align}
    if we restrict $\sqrt\eps\leq\frac12$ which, in particular, means $\eps\leq C_2\sqrt{\eps}-\eps$.
    Again we deduce two lower bounds, using also that for $x \in \Omega\setminus\gamma_\eps(\Omega)$ we even have $p_1(x)=0$ by definition of $p_1$,
    \begin{align}\label{eq:LB_p1_generic}
        \eps p_1(x) 
        &\geq 
        -\sqrt{\eps}\crho
        &&
        x\in\Omega,
        \\\label{eq:LB_p1}
        \eps p_1(x) 
        &\geq 
        - \frac{\eps}{|x|}\crho && \abs{x}\geq C_2\sqrt{\eps}-\eps.
    \end{align}
   Adding the bounds \labelcref{eq:LB_p0_generic,eq:LB_p1_generic,eq:LB_p0,eq:LB_p1} we obtain the cumulative lower bound
   \begin{align}\label{eq:LB_p_generic}
        \eps p(x) 
        &\geq 
        -\sqrt{\eps}\crho
        &&
        \text{for all }x\in\Omega,
        \\\label{eq:LB_p_sharper}
        \eps p(x) 
        &\geq 
        - \frac{\eps}{|x|}\crho && \text{if }\abs{x}\geq C_2\sqrt{\eps}.
    \end{align}

    Finally, we can now turn to proving \labelcref{eq:positive_subgrad}, making a case distinction based on $\abs{x}$.
    If $0\leq\abs{x}\leq C_2\sqrt{\eps}$ we can use the definition of $\overline w$ and the lower bound \labelcref{eq:LB_p_generic} to get
    \begin{align*}
        (\overline w(x) - d(x))\rho(x) + \eps p(x)
        &\geq 
        (C_1\sqrt{\eps} - \abs{x})\rho(x)
        +\eps p(x)
        \\
        &\geq 
        \sqrt{\eps}
        \left(C_1-C_2
        -\crho\right)\rho(x)
        \geq 0
    \end{align*}
    if we choose the gap between $C_1$ and $C_2$ sufficiently large (depending only on $\crho$).
    In the case $\abs{x}\geq C_2\sqrt{\eps}$ we can use the sharper lower bound \labelcref{eq:LB_p_sharper} to obtain
    \begin{align*}
         (\overline w(x)  - d(x))  \rho(x) + \eps p(x)
        &\geq 
        \frac{C_2(C_1-C_2)\eps}{\abs{x}}\rho(x)
        -\frac{\eps}{\abs{x}}\crho \rho(x)
        \\
        &=        
        \frac{\eps\rho(x)}{\abs{x}}
        \left(C_2(C_1-C_2)
        -
        \crho
        \right)
        \geq 0
    \end{align*}
    if we choose the gap between $C_1$ and $C_2$ sufficiently large (again depending only on $\crho$).
    Hence, we have proved \labelcref{eq:positive_subgrad} which concludes the proof.
\end{proof}
   
The first corollary of \cref{lem:cone} (in fact of \cref{rem:supersolution_cone}) is that it allows us to control the evolution of a ball under the scheme \labelcref{eq:selection}.
\begin{corollary}[Supersolution for balls]
    Under the conditions of \cref{lem:cone} there exists a constant $C>0$, depending only on  $\Lip(\rho_0)$, $\Lip(\rho_1)$, $c_\rho$, $\operatorname{diam}(\Omega)$, and the dimension $N$, such that for any $x_0\in\Omega$, $0<R<\dist(x_0,\partial\Omega)$, and $\eps>0$ sufficiently small it holds that
    \begin{align*}
        S_\eps(B(x_0,R))\supset B(x_0,R-C\sqrt{\eps}).
    \end{align*}     
\end{corollary}
\begin{proof}
    We apply \cref{rem:supersolution_cone} to $d(x):=-\abs{x-x_0}$ and note that $\sdist(x,B(x_0,R)^c)=d(x)+R$ to infer that $w_\eps^*$ in the definition of $S_\eps(B(x_0,R))$ satisfies for almost all $x\in\Omega$ that
    \begin{align*}
        w_\eps^*(x) \geq R - C\sqrt{\eps} - \abs{x-x_0}.
    \end{align*}
    Here, $C>0$ is a constant depending on $C_1,C_2$ in the definition of $\overline w$ in \cref{lem:cone}.
    Hence, we obtain $S_\eps(B(x_0,R))=\{w_\eps^*>0\}\supset B(x_0,R-C\sqrt{\eps})$.
\end{proof}
We highlight the almost-regularity of $L^2$-minimizers with Lipschitz data as an application of the cone \cref{lem:cone}. 
Our proof of consistency basically relies on the same argument, allowing us to avoid Lipschitz regularity. 
We note that in the periodic setting with constant densities, one can use a simple comparison argument to show that minimizers of the functional in \labelcref{eqn:L2+TV} (just below) are Lipschitz regular; but the moment one destroys translational invariance of the problem, such regularity becomes much more challenging. 
See for instance \cite[Theorem 3.1]{etoGiga2} for the related $\TV$ problem with boundary conditions. 
Similar almost-Lipschitz regularity results for solutions of nonlocal problems can be found, for instance, in \cite{calder2022lipschitz,bungert2023uniform}.

\begin{corollary}[Almost-Lipschitz regularity.]\label{cor:almost-Lip}
Let $\Omega\subset \R^N$ be a bounded, open, and convex set. Suppose that $w\in L^\infty(\Omega)$ satisfies
\begin{equation}\label{eqn:L2+TV}
    w := \argmin 
    \left\lbrace 
    \frac{1}{2\eps}
    \int_\Omega \abs{u-f}^2\de\rho
    +\TV_\eps(u)
    \st 
    u \in L^2(\Omega)
    \right\rbrace,
\end{equation}
for a $1$-Lipschitz function $f\in C(\Omega)$. 
Then for almost all $x,x_0\in \Omega$ it holds that
\begin{equation}\label{eqn:almost Lip}
|w(x) - w(x_0)| \leq |x-x_0|+C\sqrt{\eps},
\end{equation}
where $C>0$ depends only on $\rho$, $\operatorname{diam}(\Omega)$, and the dimension $N.$
\end{corollary}
\begin{proof}
Fix $x_0\in \Omega$. Let $w_{\rm shift}$ be the minimizer 
\begin{align*}
    w_{\rm shift} := \argmin 
    \left\lbrace 
    \frac{1}{2\eps}
    \int_\Omega \abs{u(x)-|x - x_0|}^2\de\rho(x)
    +\TV_\eps(u)
    \st 
    u \in L^2(\Omega)
    \right\rbrace,
\end{align*}
and $\overline w$ be the function defined in \cref{lem:cone}.
As $f$ is $1$-Lipschitz, we have $f(\cdot)\leq |\cdot - x_0|+f(x_0)$, and so by \cref{prop:comparison_one} applied to $w$ and $w'=w_{\rm shift} +f(x_0)$ and \cref{lem:cone} applied to $w_{\rm shift}$, we have $w\leq  w_{\rm shift} + f(x_0) \leq \overline w +f(x_0)$. Applying the same reasoning with supersolutions (using \cref{rem:supersolution_cone}) and noting that $\overline w(\cdot) \leq |\cdot - x_0| +C \sqrt{\eps}$ gives that
$$ |w(x) -f(x_0)| \leq |x - x_0| +C \sqrt{\eps}$$
for any $x\in \Omega$. Using this inequality twice (once with $x = x_0$), applying the triangle inequality, and increasing $C$ directly gives \labelcref{eqn:almost Lip}.
\end{proof}

Finally, we conclude the proof of \cref{thm:monotone and consistent} by showing that the operator is consistent.
  
\begin{proposition}[Consistency]\label{prop:consist}
     The operator $S_\eps$ defined in \labelcref{eq:selection_operator} is consistent in the sense of \cref{def:consist}.
\end{proposition}
\begin{proof}
The proof follows the strategy of \cite[Proposition 4.1]{chambolle2007approximation} with non-trivial modifications since the scheme \labelcref{eq:selection_operator} involves the nonlocal total variation instead of the local one. We show that  \cref{def:consist} is satisfied for subflows, with superflows being analogous.

\emph{Step 1 (Construction of an variational subsolution).} We let $[t_0,t_1]\ni t \mapsto A(t)$ be a subflow in the sense of \cref{def:superflow} contained in the neighborhood $B$. Recall that we define $d(x,t):=\sdist(x,A^c(t))$.

For fixed $t\in[t_0,t_1]$ and $r>0$ we define $$\tube := \{\abs{d(\cdot,t)}<r\}$$ to be the tube of width $r$ around the boundary $\partial A(t)$.
Here we choose $r$ sufficiently small such that $\closure{\tube}\subset B\cap(\R^N \times \{t\})$, so that $d(x,t+\tau)$ is in $C^{2,1}_{x,\tau}(\tube \times [-\eps,\eps])$ for all small $\eps$; note the choice of $r$ can be made independent of $t$ depending only on the smooth subflow.
Let $\psi:\R\to\R$ be smooth with $\psi(s)\geq s$, $\psi(s)=s$ for $s$ in a neighborhood of $0$, and $\psi'(s)\geq c>0$, and define $v_\eps(x) := \psi(d(x,t+\eps))$.
By \cref{def:superflow}, it holds that
\begin{align*}
    \frac{v_\eps(x)-d(x,t)}{\eps} 
    &\geq 
    \frac{d(x,t+\eps)-d(x,t)}{\eps}
    \\
    &=
    \intbar_0^\eps \frac{\de}{\de\tau}d(x,t+\tau)\de\tau
    \\
    &=
    \intbar_0^\eps \partial_t d(x,t+\tau)\de\tau
    \\
    &\geq 
    \intbar_0^\eps \frac{1}{\rho(x)}\div\left(\rho(x)\grad d(x,t+\tau)\right)\de\tau + \delta.
\end{align*}
Letting $\omega$ denote a modulus of continuity of $\frac{1}{\rho(x)}\div\left(\rho(x)\grad d(x,\tau)\right)$ in $\tau$ (uniformly in $x$) on $B$, we have
\begin{align}\label{eq:diff_ineq_1}
    \frac{v_\eps(x)-d(x,t)}{\eps} 
    \geq 
    \frac{1}{\rho(x)}\div\left(\rho(x)\grad d(x,t+\eps)\right) + \delta - \omega(\eps).
\end{align}
Note furthermore that $\grad v_\eps(x) = \psi'(d(x,t+\eps))\grad d(x,t+\eps)$.
On one hand this implies
\begin{align}\label{eq:gradient_bound_v}
    \abs{\nabla v_\eps}\geq c \quad \text{ in } \tube,
\end{align}
which will be useful later.
On the other hand, we see that
\begin{align*}
    \frac{\grad v_\eps(x)}{\abs{\grad v_\eps(x)}} = \frac{\grad d(x,t+\eps)}{\abs{\grad d(x,t+\eps)}} = \grad d(x,t+\eps).  
\end{align*}
Using this and reordering \labelcref{eq:diff_ineq_1} we get
\begin{align}\label{eq:diff_ineq_v_continuum}
    \frac{v_\eps(x)-d(x,t)}{\eps} 
    \rho(x)
    -
    \div\left(\rho(x)\frac{\grad v_\eps(x)}{\abs{\grad v_\eps(x)}}\right) 
    -\rho(x)\left(\delta - \omega(\eps)\right)
    \geq 0.
\end{align}
Let $\varphi \in L^\infty(\tube)$ be a non-negative test function with $\supp\varphi\subset \tube_{2\eps}$.
Let us also define the energy
\begin{align*}
    E_\eps(u;\tube) := \frac{1}{2\eps}\int_{\tube}\abs{u-d(\cdot,t)}^2\de\rho + \TV_\eps(u;\tube),
\end{align*}
where $\TV_\eps(u;\tube)$ denotes the total variation \labelcref{eq:TV} with $\Omega$ replaced by $\tube$. We let 
\begin{align*}
    V_\eps(y):=\frac{\de}{\de \mathcal{L}^N}\left[\frac{(\Gamma_\eps)_\sharp\rho_0 - \rho_0}{\eps} + \frac{\rho_1-(\gamma_\eps)_\sharp\rho_1}{\eps}\right]
\end{align*}
be the density of the pushforward, where the right-hand side is defined as in \cref{prop:density_of_pushforward} with $u$ replaced by $v_\eps$, which is admissible due to \labelcref{eq:gradient_bound_v}.
Multiplying \labelcref{eq:diff_ineq_v_continuum} by $\varphi$, using its non-negativity, and integrating over $\tube$ yields
\begin{align*}
    E_\eps(v_\eps;\tube)
    &\leq 
    E_\eps(v_\eps;\tube)
    +
    \int_{\tube}
    \left(
    \frac{v_\eps(x)-d(x,t)}{\eps} 
    \rho(x)
    -
    \div\left(\rho(x)\frac{\grad v_\eps(x)}{\abs{\grad v_\eps(s)}}\right) 
    \right)
    \varphi(x)
    \de x
    \\
    & \qquad  -
    (\delta-\omega(\eps))
    \int_{\tube} \varphi(x)\de\rho(x)
    \\
    &=
    E_\eps(v_\eps;\tube)
    +
    \int_{\tube}
    \left(
    \frac{v_\eps(x)-d(x,t)}{\eps} 
    \rho(x)
    +
    V_\eps(x)
    \right)
    \varphi(x)
    \de x
    \\
    &\qquad
    -
    \int_{\tube}
    \left(
    V_\eps(x)
    +
    \div\left(\rho(x)\frac{\grad v_\eps(x)}{\abs{\grad v_\eps(s)}}\right) 
    \right)
    \varphi(x)\de x
    \\
    &\qquad
    -
    (\delta-\omega(\eps))
    \int_{\tube} \varphi(x)\de\rho(x)
    \\
    &\leq 
    E_\eps(v_\eps + \varphi;\tube)
    +
    (\omega(\eps)+o_{\eps\to 0}(1) - \delta)
    \int_{\tube} \varphi(x)\de\rho(x),
\end{align*}
where in the last step, we completed a square, used \cref{prop:density_of_pushforward,prop:consistency} on $\tube$ together with the gradient bound \labelcref{eq:gradient_bound_v} for $v_\eps$ and the fact that $\supp\varphi\subset\tube_{2\eps}$.
Since $\delta>0$ we can choose $\eps>0$ sufficiently small such that the second term is non-positive which implies 
\begin{equation}\label{eqn:supersoln}
E_\eps(v_\eps;\Omega')\leq E_\eps(v_\eps+\varphi;\Omega') 
\end{equation} f
or $\eps>0$ sufficiently small.

\emph{Step 2 (Conclusion, assuming ordered boundary values).} Supposing that
\begin{equation}\label{ass:orderedBC}
v_\eps \geq w_\eps^*  \quad \text{ on }\tube\setminus \tube_{2\eps},
\end{equation}
the non-negative test function $\varphi_\eps := v_\eps\vee w_\eps^* - v_\eps$ (where $w_\eps^*$ solves the scheme \labelcref{eq:selection_operator}) can be inserted into \labelcref{eqn:supersoln}.
Consequently, we have that $E_\eps(v_\eps)\leq E_\eps(v_\eps\vee w_\eps^*)$, and \cref{prop:comparison_two} implies that $v_\eps\geq w_\eps^*$ on $\tube$ (one must technically deal with null-sets, but this may be done as in the proof of \cref{prop:monotone}), and hence, we find
\begin{align*}
    S_\eps(A(t)) \cap \tube =\{w_\eps^\ast > 0\}\cap \tube\subset\{v_\eps > 0\}\cap \tube=\{d(\cdot,t+\eps)> 0\}\cap \tube = A(t+\eps)\cap \tube.
\end{align*}
 Similarly, we will see in the next step,
\begin{equation}\label{eqn:sign estimate}
\{w_\eps^\ast > 0\}\setminus \tube \subset \{d(\cdot,t)> 0\}\setminus \tube.
\end{equation}
Further, outside of the set $\tube$, for sufficiently small $\eps$ (depending only on the smooth subflow), we have $A(t)\setminus \tube = A(t+\eps)\setminus \tube$.
Putting these last two pieces together, we recover
$$ S_\eps(A(t)) \setminus \tube =\{w_\eps^\ast > 0\}\setminus \tube \subset \{d(\cdot,t)> 0\}\setminus \tube = A(t+\eps)\setminus \tube.$$
Uniting the subset relations for $S_\eps(A(t))$ concludes the proof.
We now turn to the proof of \labelcref{ass:orderedBC,eqn:sign estimate}.

\emph{Step 3 (Ordered boundary values).}
To prove \labelcref{ass:orderedBC}, we have to pick a suitable function $\psi$ in the definition of $v_\eps = \psi(d(\cdot,t+\eps))$.
So far we have only used that $\psi(s)\geq s$, $\psi(s)=s$ in a neighborhood of $0$, and that $\psi' \geq c>0$. 

First, we argue that one can find $\psi$ such that $v_\eps\geq w_\eps^*$ in the part of the $2\eps$-neighborhood of the boundary of $\tube$ that lies inside of $A(t)$:
For all $\eps>0$ small enough and $x\in\partial \tube$ with $d(x,t)=r$, it holds that $d(x,t+\eps)\geq \frac{7 r}{8}$ since $d$ is uniformly continuous in time.
Since $d$ is also uniformly continuous in space, we get $d(x,t+\eps)\geq\frac{3r}{4}$ for all $x$ in $(\tube \setminus \tube_{2\eps}) \cap A(t)$ if $\eps>0$ is sufficiently small.
On the other hand, by \cref{prop:existence_and_Lipschitz} it holds that $w_\eps^*\leq\norm{d(\cdot,t)}_{L^\infty(\Omega)}\leq\operatorname{diam}\Omega$. 
So if we choose $\psi$ such that $\psi(\frac{3r}{4})\geq \operatorname{diam}\Omega$, we get 
\begin{align*}
    v_\eps(x) = \psi(d(x,t+\eps))\geq \psi\left(\frac{3r}{4}\right)\geq \operatorname{diam}\Omega \geq w_\eps^*(x)
\end{align*}
for all $x\in (\tube \setminus \tube_{2\eps}) \cap A(t)$.

Next, we argue that also in the $2\eps$-neighborhood of the exterior part of the boundary of $\tube$ one can find an appropriate $\psi$ such that $v_\eps\geq w_\eps^*$:
The argument for this is more involved than for the inner part since in principle $w_\eps^*$ could be arbitrarily close to zero outside of $A(t)$ whereas the signed distance function $d(\cdot,t+\eps)$ in the definition of $v_\eps$ might be substantially negative. If this happened, to obtain $v_\eps\geq w_\eps^*$ we could be forced to take $\psi(-3r/4) = 0$ breaking the constraint $\psi' \geq c$.

Instead, fix a point $x\in\partial \tube$ such that $d(x,t)=-r$.
Since the distance function is $1$-Lipschitz 
\begin{align*}
    d(\cdot,t) \leq \abs{\cdot-x} - r,
\end{align*}
and therefore \cref{prop:comparison_one} ensures that $w_\eps^*\leq w_\eps$ in $\Omega$ where 
\begin{align*}
    w_\eps := \argmin_{u\in L^2(\Omega)}\frac{1}{2\eps}\int_\Omega\abs{u-(\abs{\cdot-x} - r)}^2\de\rho+\TV_\eps(u).
\end{align*}
However, by \cref{lem:cone}, $w_\eps \leq \overline w - r$, and it follows that $w_\eps^* \leq -r + C_1\sqrt{\eps}$ for all $|x' - x|\leq C_2\sqrt{\eps}$. Noting that this reasoning can be uniformly applied at all points $x\in\partial \tube$ with $d(x,t)=-r$, we see that for sufficiently small $\eps>0$, 
\begin{equation}\label{eqn:almost sign estimate}
 w_\eps^* \leq -\frac{3}{4}r \quad \text{ on } \quad (\tube \setminus \tube_{2\eps}) \setminus A(t).
 \end{equation}
Restricting $\psi$ to satisfy $\psi(t)\geq - \frac{3r}{4}$ for any $t\geq -\operatorname{diam}\Omega$ we therefore get
\begin{align*}
    v_\eps(x) = \psi(d(x,t+\eps)) \geq - \frac{3r}{4}
    \geq w_\eps^*(x)
\end{align*}
for all $x\in (\tube \setminus \tube_{2\eps}) \setminus A(t)$.
Hence, we have shown \labelcref{ass:orderedBC}.
The same reasoning used to obtain \labelcref{eqn:almost sign estimate}, but now at a point for which $d(x,t)\leq-r$, shows
$$\{w_\eps^\ast > 0\}\setminus (\tube\cup A(t)) \subset \{d(\cdot,t)> 0\}\setminus (\tube\cup A(t)) =\emptyset $$
which directly gives \labelcref{eqn:sign estimate}, completing the proof.

Note we have proven that there exists an $\eps_0>0$ sufficiently small such that consistency in \cref{def:consist} is satisfied at a given time $t$ for all $\eps <\eps_0$, but actually, our estimate for $\eps_0$ is uniform in $t\in [t_0,t_1]$.
\end{proof}

\section*{Acknowledgments}

The authors would like to thank Antonin Chambolle for fruitful discussions around the construction subsolutions for cone data which happened during the Oberwolfach workshop 2349 ``Variational Methods for Evolution''.
Parts of this work were done when LB was affiliated with the Technical University of Berlin, supported by Germany’s Excellence Strategy – The Berlin Mathematics Research Center MATH+ (EXC-2046/1, project ID: 390685689).
The authors were affiliated with the Hausdorff Center for Mathematics during parts of this project and the funding from the Deutsche Forschungsgemeinschaft (DFG, German Research Foundation) 
under Germany's Excellence Strategy -- EXC-2047/1 -- 390685813 is greatly appreciated. KS was also supported by the DFG project 211504053 - SFB 1060.

\printbibliography

\end{document}